\documentclass[12pt]{amsart}
\usepackage{amssymb,amsthm,amsmath,amsfonts,latexsym,tikz,hyperref,cleveref,start4,mathtools,xcolor,shuffle,microtype,enumitem,breakurl}
\usepackage[hmargin=1in,vmargin=1in]{geometry}

\setlist{itemsep=3bp}

\newtheorem{thm}{Theorem}[section]

\newtheorem{cor}[thm]{Corollary}
\newtheorem{lem}[thm]{Lemma}

\newtheorem{question}[thm]{Question}
\newtheorem{problem}[thm]{Problem}

\DeclareMathOperator{\cyc}{cyc}
\DeclareMathOperator{\QSym}{QSym}
\DeclareMathOperator{\cQSym}{cQSym}
\DeclareMathOperator{\cst}{cst}
\DeclareMathOperator{\cDes}{cDes}
\DeclareMathOperator{\Pk}{Pk}
\DeclareMathOperator{\cPk}{cPk}
\DeclareMathOperator{\cdes}{cdes}
\DeclareMathOperator{\pk}{pk}
\DeclareMathOperator{\cpk}{cpk}
\DeclareMathOperator{\cmaj}{cmaj}

\DeclareMathOperator{\Comp}{Comp}
\DeclareMathOperator{\cComp}{cComp}

\DeclareMathOperator{\ocmaj}{ocmaj}
\DeclareMathOperator{\Lpk}{Lpk}
\DeclareMathOperator{\lpk}{lpk}
\DeclareMathOperator{\Val}{Val}
\DeclareMathOperator{\val}{val}
\DeclareMathOperator{\Ddes}{Ddes}
\DeclareMathOperator{\ddes}{ddes}
\DeclareMathOperator{\Rpk}{Rpk}
\DeclareMathOperator{\rpk}{rpk}
\DeclareMathOperator{\Epk}{Epk}
\DeclareMathOperator{\epk}{epk}
\DeclareMathOperator{\br}{br}
\DeclareMathOperator{\udr}{udr}
\DeclareMathOperator{\cbr}{cbr}
\DeclareMathOperator{\cVal}{cVal}
\DeclareMathOperator{\cval}{cval}
\DeclareMathOperator{\comaj}{comaj}
\DeclareMathOperator{\ccomaj}{ccomaj}
\DeclareMathOperator{\ost}{ost}

\newcommand{\shu}{\shuffle}

\begin{document}
\pagestyle{plain}

\title{Cyclic shuffle-compatibility via cyclic shuffle algebras}
\author{Jinting Liang}
\address{Department of Mathematics,  Michigan State University,
 East Lansing, MI 48824, USA}
\email{liangj26@msu.edu}
\author{Bruce E. Sagan}
\address{Department of Mathematics, Michigan State University,
 East Lansing, MI 48824, USA}
\email{sagan@math.msu.edu}
\author{Yan Zhuang}
\address{Department of Mathematics and Computer Science, Davidson College,
Davidson, NC 28035, USA}
\email{yazhuang@davidson.edu}
\thanks{YZ was partially supported by an AMS-Simons Travel Grant and NSF grant DMS-2316181.}

%\author{}
%\address{}
%\email{}

\date{\today}

%     https://mathscinet.ams.org/mathscinet/msc/msc2020.html

 
\begin{abstract}
A permutation statistic $\st$ is said to be shuffle-compatible if the distribution of $\st$ over the set of shuffles of two disjoint permutations $\pi$ and $\sigma$ depends only on $\st\pi$, $\st\sigma$, and the lengths of $\pi$ and $\sigma$. Shuffle-compatibility is implicit in Stanley's early work on $P$-partitions, and was first explicitly studied by Gessel and Zhuang, who developed an algebraic framework for shuffle-compatibility centered around their notion of the shuffle algebra of a shuffle-compatible statistic. For a family of statistics called descent statistics, these shuffle algebras are isomorphic to quotients of the algebra of quasisymmetric functions.

Recently, Domagalski, Liang, Minnich, Sagan, Schmidt, and Sietsema defined a version of shuffle-compatibility for statistics on cyclic permutations, and studied cyclic shuffle-compatibility through purely combinatorial means. In this paper, we define the cyclic shuffle algebra of a cyclic shuffle-compatible statistic, and develop an algebraic framework for cyclic shuffle-compatibility in which the role of quasisymmetric functions is replaced by the cyclic quasisymmetric functions recently introduced by Adin, Gessel, Reiner, and Roichman. We use our theory to provide explicit descriptions for the cyclic shuffle algebras of various cyclic permutation statistics, which in turn gives algebraic proofs for their cyclic shuffle-compatibility.
\end{abstract}

\maketitle

\tableofcontents
%%%%%%%%%%%%%%%%%%%%%%%%%%%%%%%%
%
% 	
%
%%%%%%%%%%%%%%%%%%%%%%%%%%%%%%%%

%\bibliographystyle{plain}

%\nocite{*}
%\bibliographystyle{abbrvnat}

\section{Introduction}

We say that $\pi = \pi_1 \pi_2 \cdots \pi_n$ is a (\textit{linear}) \textit{permutation} of \textit{length} $n$ if it is a sequence of $n$ distinct letters---not necessarily from 1 to $n$---in $\bbP$, the set of positive integers. (We refer to these as linear permutations to distinguish them from cyclic permutations, but we will often drop the descriptor ``linear'' if it is clear from context that we are referring to linear permutations.) For example, $826491$ is a permutation of length 6. Let $|\pi|$ denote the length of a permutation $\pi$, let $\mathfrak{P}_n$ denote the set of permutations of length $n$, and $\mathfrak{S}_n \subseteq \mathfrak{P}_n$ the set of permutations of $[n]\coloneqq \{1, 2, \dots, n \}$. Note that $\mathfrak{P}_0$ and $\mathfrak{S}_0$ consist only of the empty word.

Let $\pi\in\mathfrak{P}_{m}$ and $\sigma\in\mathfrak{P}_{n}$ be \textit{disjoint} permutations, that is, permutations with no letters in common. We say that $\tau\in\mathfrak{P}_{m+n}$ is a \textit{shuffle}
of $\pi$ and $\sigma$ if both $\pi$ and $\sigma$ are subsequences of $\tau$. The set of shuffles of $\pi$ and $\sigma$ is denoted $\pi \shuffle \sigma$. For example, 
$$
71 \shuffle 25 = \{ 7125, 7215, 7251, 2715, 2751, 2571 \}.
$$

Following \cite{Gessel2018}, a (\textit{linear}) \textit{permutation statistic} is a function $\st$ on permutations such that $\st \pi = \st \sigma$ whenever $\pi$ and $\sigma$ have the same relative order.\footnote{The \textit{standardization} of a permutation $\pi \in \mathfrak{P}_n$ is the permutation in $\mathfrak{S}_n$ obtained by replacing the smallest letter in $\pi$ by 1, the second smallest by 2, and so on. Then two permutations are said to \textit{have the same relative order} if they have the same standardization.} Three classical permutation statistics, dating back to MacMahon \cite{macmahon}, are the descent set $\Des$, descent number $\des$, and the major index $\maj$. We say that $i\in[n-1]$ is a \textit{descent} of $\pi\in\mathfrak{P}_n$ if $\pi_{i}>\pi_{i+1}$. The \textit{descent set} of $\pi$
\[
\Des \pi \coloneqq\{\,i\in[n-1] : \pi_{i}>\pi_{i+1}\,\}
\]
is the set of its descents, the \textit{descent number}
\[
\des \pi \coloneqq\left|\Des \pi \right|
\]
its number of descents, and the \textit{major index} 
\[
\maj \pi \coloneqq\sum_{i\in\Des \pi}i
\]
the sum of its descents.

Several other permutation statistics---somewhat less classical but still well-studied---are based on the notion of peaks. We say that $i\in \{2,3,\dots,n-1\}$ is a \textit{peak} of $\pi\in\mathfrak{P}_n$ if $\pi_{i-1} < \pi_{i} > \pi_{i+1}$. The \textit{peak set} of $\pi$
\[
\Pk \pi \coloneqq\{\,i\in \{2,3,\dots,n-1\} : \pi_{i-1} < \pi_{i} > \pi_{i+1} \,\}
\]
is the set of its peaks, and the \textit{peak number}
\[
\pk \pi \coloneqq\left|\Pk \pi \right|
\]
is its number of peaks. 
Some related statistics, such as the left peak set and left peak number, will be defined in Section \ref{ss-otherdes}.

Given a set $S$ of permutations and a permutation statistic $\st$, the \textit{distribution} of $\st$ over $S$ is the multiset 
$$ \st S \coloneqq \{\{\,\st \pi : \pi \in S \,\}\}$$
of all values of $\st$ among permutations in $S$, including multiplicity. For instance, 
$$
\des \mathfrak{S}_3 = \{\{ 0,1^4,2 \}\};
$$
among the six permutations in $\mathfrak{S}_3$, only 123 has no descents, only 321 has two descents, and the other four have one descent each. 

All of the statistics defined above have a remarkable property related to shuffles, called ``shuffle-compatibility''. 
We say that $\st$ is \textit{shuffle-compatible} if the distribution of $\st$ over the shuffles of any two disjoint permutations $\pi$ and $\sigma$ depends only on $\st\pi$, $\st\sigma$, and the lengths of $\pi$ and $\sigma$. In other words, $\st$ is shuffle-compatible if $\st(\pi \shuffle \sigma) = \st(\pi^\prime \shuffle \sigma^\prime)$ whenever $\st\pi = \st\pi^\prime$, $\st\sigma = \st \sigma^\prime$, $|\pi| = |\pi^\prime|$, and $|\sigma| = |\sigma^\prime|$. 

Shuffle-compatibility dates back to the early work of Stanley, as the shuffle-compatibility of the descent set, descent number, and major index are implicit consequences of the theory of $P$-partitions \cite{sta:osp}. Likewise, Stembridge's work on enriched $P$-partitions imply that the peak set and peak number are shuffle-compatible. Gessel and Zhuang coined the term ``shuffle-compatibility'' and initiated the study of shuffle-compatibility per se in 2018; in \cite{Gessel2018}, they developed an algebraic framework for shuffle-compatibility centered around the notion of the shuffle algebra of a shuffle-compatible permutation statistic, which is well-defined if and only if the statistic is shuffle-compatible and whose multiplication encodes the distribution of the statistic over sets of shuffles. 

Gessel's \cite{Gessel1984} quasisymmetric functions serve as natural generating functions for $P$-partitions, and for a special family of statistics called ``descent statistics'', one can use quasisymmetric functions to characterize shuffle algebras and prove shuffle-compatibility results. Notably, the multiplication rule for fundamental quasisymmetric functions shows that the descent set is shuffle-compatible and that its shuffle algebra is isomorphic to the algebra $\QSym$ of quasisymmetric functions. One of Gessel and Zhuang's main results is a necessary and sufficient condition for shuffle-compatibility of descent statistics which implies that the shuffle algebra of any shuffle-compatible descent statistic is isomorphic to a quotient algebra of $\QSym$.

In the past few years, shuffle-compatibility has become an active topic of research; see \cite{Adin2021, BakerJarvis2020, Detal:csc, Grinberg2018, KantarciOguz2018, 2209.00051, Moustakas2022, Yang2022, Zhuang2021} for a selection of references. Most relevant to our present work are the recent papers of Adin--Gessel--Reiner--Roichman \cite{Adin2021} and Liang \cite{2209.00051} on cyclic quasisymmetric functions and toric $[\vec{D}]$-partitions, and of Domagalski--Liang--Minnich--Sagan--Schmidt--Sietsema \cite{Detal:csc} which defined and studied a notion of shuffle-compatibility for cyclic permutations.

\subsection{Cyclic permutations, statistics, and shuffles}

Given a linear permutation $\pi = \pi_1 \pi_2 \cdots \pi_n$, let $[\pi]$ be the equivalence class of $\pi$ under cyclic rotation, that is, 
\[
[\pi]\coloneqq\{\pi_{1}\pi_{2}\cdots\pi_{n},\ \pi_n\pi_{1}\cdots\pi_{n-1},\
\dots,\ \pi_{2}\cdots\pi_{n}\pi_1\}.
\]
The sets $[\pi]$ are called \textit{cyclic permutations}. 
The \textit{length} of a cyclic permutation $[\pi]$ refers to the length of $\pi$, which makes sense because all linear permutation representatives of $[\pi]$ have the same length. 
For example,
$$
[168425]=\{168425, 516842, 251684, 425168, 842516, 684251\}
$$
has length $6$.

In analogy to linear permutation statistics, let us define a \textit{cyclic permutation statistic} to be a function $\cst$ on cyclic permutations such that $\cst[\pi] = \cst[\sigma]$ whenever $\pi$ and $\sigma$ have the same relative order. Two examples of cyclic permutation statistics are the cyclic descent set $\cDes$ and the cyclic descent number $\cdes$. First, define the \textit{cyclic descent set} of a linear permutation $\pi\in \mathfrak{P}_n$ by
\[
\cDes \pi \coloneqq\{\,i\in[n] : \pi_{i}>\pi_{i+1}\text{ where }i\text{ is considered modulo }n\,\};
\]
the elements of $\cDes\pi$ are called \textit{cyclic descents} of $\pi$. The \textit{cyclic descent set} of a cyclic permutation $[\pi]$ is the multiset 
$$\cDes[\pi] \coloneqq \{\{\,\cDes\bar{\pi}:\bar{\pi}\in[\pi]\,\}\},$$
i.e., the distribution of the linear statistic $\cDes$ over all linear permutation representatives of $[\pi]$. For example, we have
$$
\cDes[168425] = \{\{\, \{3,4,6\},\ \{1,4,5\},\ \{2,5,6\},\ \{1,3,6\},\ \{1,2,4\},\ \{2,3,5\} \,\}\},
$$
and 
$$
\cDes[279358] = \{\{\, \{3,6\}^2, \{1,4\}^2, \{2,5\}^2 \,\}.
$$
 Note that $\cDes[\pi]$ can also be characterized as the multiset of cyclic shifts of $\cDes\pi$. More precisely, given $S\subseteq [n]$ and an integer $i$, define the \textit{cyclic shift} $S+i$ by 
$$ S+i \coloneqq \{\, s+i : s\in S \,\}$$
where the values are considered modulo $n$; then
$$\cDes[\pi] = \{\{\,\cDes\pi+i : i\in[n]\,\}\}.$$

The \textit{cyclic descent number} of a linear permutation $\pi$ is given by
\[
\cdes \pi \coloneqq\left|\cDes \pi \right|,
\]
and we can then define the \textit{cyclic descent number} of a cyclic permutation $[\pi]$ by
\[
\cdes[\pi] \coloneqq \cdes\pi,
\]
which is well-defined because all linear permutations in $[\pi]$ have the same number of cyclic descents. The cyclic peak set $\cPk$ and cyclic peak number $\cpk$ can be defined in an analogous way, and we will state their definitions in Section \ref{ss-cPkcpk}. On the other hand, finding a suitable cyclic analogue of the major index statistic is challenging; we will address this in Section \ref{ss-cmaj}. 

Given disjoint $\pi \in \mathfrak{P}_m$ and $\sigma \in\mathfrak{P}_n$, we say that $[\tau]$ is a \textit{cyclic shuffle} of $[\pi]$ and $[\sigma]$ if $\tau\in \mathfrak{P}_{m+n}$ and there exist $\bar{\pi} \in [\pi]$ and $\bar{\sigma} \in [\sigma]$ such that $\tau$ is a (linear) shuffle of $\bar{\pi}$ and $\bar{\sigma}$. Let $[\pi]\shuffle[\sigma]$ denote the set of cyclic shuffles of $[\pi]$ and $[\sigma]$. For instance, we have 
$$
[63]\shuffle[24]=\{[6324],\ [6234],\ [6243],\ [6342],\ [6432],\ [6423]\}.
$$

A cyclic permutation statistic $\cst$ is called \textit{cyclic shuffle-compatible} if the distribution of $\cst$ over all cyclic shuffles of $[\pi]$ and $[\sigma]$ depends only on $\cst[\pi]$, $\cst[\sigma]$, and the lengths of $[\pi]$ and $[\sigma]$. That is, $\cst$ is cyclic shuffle-compatible if we have $\cst([\pi]\shuffle[\sigma])=\cst([\pi^\prime]\shuffle[\sigma^\prime])$ whenever $\cst[\pi]=\cst[\pi^\prime]$, $\cst[\sigma]=\cst[\sigma^\prime]$, $|\pi|=|\pi^\prime|$, and $|\sigma|=|\sigma^\prime|$.

The first results in cyclic shuffle-compatibility were implicit in the work of Adin et al.\ \cite{Adin2021}, which introduced toric $[\vec{D}]$-partitions (a toric poset analogue of $P$-partitions) and cyclic quasisymmetric functions (which are natural generating functions for toric $[\vec{D}]$-partitions). In particular, Adin et al.\ established a multiplication formula for fundamental cyclic quasisymmetric functions which implies that the cyclic descent set $\cDes$ is cyclic shuffle-compatible, and they also proved the formula 
\[
\sum_{[\tau]\in[\pi]\shuffle[\sigma]}q^{\cdes\tau}=(1-q)^{\left|\pi\right|+\left|\sigma\right|}\sum_{k=0}^{\infty}{k+\left|\pi\right|-\cdes\pi-1 \choose \left|\pi\right|-1}{k+\left|\sigma\right|-\cdes\sigma-1 \choose \left|\sigma\right|-1}kq^{k}
\]
which implies that the cyclic descent number $\cdes$ is cyclic shuffle-compatible. 

In \cite{Detal:csc}, Domagalski et al.\ formally defined cyclic shuffle-compatibility and proved a result called the ``lifting lemma,'' which allows one (under certain nice conditions) to prove that a cyclic statistic is cyclic shuffle-compatible from the shuffle-compatibility of a related linear statistic. They then used the lifting lemma to prove the cyclic shuffle-compatibility of all four statistics $\cDes$, $\cdes$, $\cPk$, and $\cpk$.

Most recently, Liang \cite{2209.00051} defined and studied enriched toric $[\vec{D}]$-partitions, an analogue of enriched $P$-partitions for toric posets, whose generating functions are ``cyclic peak quasisymmetric functions''. She derived a multiplication formula for these cyclic peak quasisymmetric functions which gives a different proof for the cyclic shuffle-compatibility of the cyclic peak set $\cPk$.

The lifting lemma of Domagalski et al.\ is purely combinatorial, but the work of Adin et al.\ and Liang suggest that there is an algebraic framework for cyclic shuffle-compatibility \`a la Gessel and Zhuang, in which the role of quasisymmetric functions is replaced by cyclic quasisymmetric functions. The goal of our paper is to develop this algebraic framework.

See \cite{LSZ-ea} for an extended abstract of this work.

\subsection{Outline}

The organization of this paper is as follows. In Section \ref{s-shufalg}, we review Gessel and Zhuang's definition of the shuffle algebra of a shuffle-compatible permutation statistic, and then we define the cyclic shuffle algebra of a cyclic shuffle-compatible statistic. We prove several general results about cyclic shuffle-compatibility via cyclic shuffle algebras, including a result (Theorem \ref{t-AtoAcyc}) allowing one to construct cyclic shuffle algebras from linear ones.

In Section \ref{s-scqsym}, we review the role of quasisymmetric functions in the theory of (linear) shuffle-compatibility, and then we develop an analogous theory concerning cyclic quasisymmetric functions and cyclic shuffle-compatibility. We use Theorem \ref{t-AtoAcyc} to construct the non-Escher subalgebra $\cQSym^-$ of cyclic quasisymmetric functions from the algebra $\QSym$ of quasisymmetric functions, which gives another proof that $\cDes$ is cyclic shuffle-compatible and shows that the cyclic shuffle algebra of $\cDes$ is isomorphic to $\cQSym^-$. We then give a necessary and sufficient condition for cyclic shuffle-compatibility of cyclic descent statistics which implies that the cyclic shuffle algebra of any cyclic shuffle-compatible cyclic descent statistic is isomorphic to a quotient algebra of $\cQSym^-$. 

In Section \ref{s-charcsa}, we use the theory developed in Section \ref{s-scqsym} to give explicit descriptions of the shuffle algebras of the statistics $\cPk$, $\cpk$ $\cdes$, and $(\cpk,\cdes)$ which in turn yields algebraic proofs for their cyclic shuffle-compatibility.

In Section \ref{s-induced}, we define a family of multiset-valued cyclic statistics induced from linear statistics, and investigate cyclic shuffle-compatibility for some of these statistics. This approach yields a definition of a cyclic major index which  is different from the one proposed earlier by Ji and Zhang \cite{Ji2022}; unfortunately, neither of these cyclic major index statistics are cyclic shuffle-compatible.

We conclude the paper in Section \ref{s-openprob} with a discussion of open problems and questions related to our work.

\section{Cyclic shuffle algebras} \label{s-shufalg}

At the heart of Gessel and Zhuang's algebraic framework for shuffle-compatibility is the notion of a shuffle algebra. In this section, we review the definition of the shuffle algebra of a shuffle-compatible (linear) permutation statistic, define a cyclic analogue of shuffle algebras for cyclic shuffle-compatible statistics, and prove several general results about cyclic shuffle-compatibility through cyclic shuffle algebras, including one that can be used to construct cyclic shuffle algebras from shuffle algebras of linear permutation statistics.

\subsection{Definitions}

Let $\st$ be a permutation statistic. We say that $\pi$ and $\sigma$ are $\st$-\textit{equivalent} if $\st \pi=\st \sigma$ and $\left|\pi\right|=\left|\sigma\right|$. In this way, every permutation statistic induces an equivalence relation on permutations, and we write the $\st$-equivalence class of $\pi$ as $\pi_{\st}$.\footnote{In \cite{Gessel2018}, the authors write $[\pi]_{\st}$ for the $\st$-equivalence class of $\pi$, but here we will use this notation for $\st$-equivalence classes of cyclic permutations in place of the more cumbersome $[[\pi]]_{\st}$.}

Let $\mathcal{A}_{\st}$ denote the $\mathbb{Q}$-vector space consisting of formal linear combinations of $\st$-equivalence classes of permutations. If $\st$ is shuffle-compatible, then we can turn $\mathcal{A}_{\st}$ into a $\mathbb{Q}$-algebra by endowing it with the multiplication
\[
\pi_{\st}\sigma_{\st}=\sum_{\tau\in \pi\shuffle\sigma}\tau_{\st}
\]
for any disjoint representatives $\pi \in \pi_{\st}$ and $\sigma \in \sigma_{\st}$; this multiplication is well-defined (i.e., the choice of $\pi$ and $\sigma$ does not matter) precisely when $\st$ is shuffle-compatible. The $\mathbb{Q}$-algebra $\mathcal{A}_{\st}$ is called the (\textit{linear}) \textit{shuffle algebra} of $\st$. Observe that ${\mathcal A}_{\st}$ is graded by length, that is, $\pi_{\st}$ belongs to the $n$th homogeneous component of ${\mathcal A}_{\st}$ if $\pi$ has length $n$.

Our definition of cyclic shuffle algebras will be analogous to that of linear ones. Let $\cst$ be a cyclic permutation statistic. Then the cyclic permutations $[\pi]$ and $[\sigma]$ are called \textit{$\cst$-equivalent} if $\cst[\pi]=\cst[\sigma]$ and $\left|\pi\right|=\left|\sigma\right|$, and we use the notation $[\pi]_{\cst}$ to denote the $\cst$-equivalence class of the cyclic permutation $[\pi]$. We associate to $\cst$ a $\mathbb{Q}$-vector space ${\mathcal A}_{\cst}^{\cyc}$ by taking as a basis the set of all $\cst$-equivalence classes of permutations, and then we give this vector space a multiplication by defining 
\[
[\pi]_{\cst}[\sigma]_{\cst}=\sum_{[\tau]\in [\pi]\shu[\sigma]}[\tau]_{\cst}
\]
for any disjoint $\pi$ and $\sigma$ with $[\pi] \in [\pi]_{\cst}$ and $[\sigma] \in [\sigma]_{\cst}$; this multiplication is well-defined if and only if $\cst$ is cyclic shuffle-compatible. The resulting $\mathbb{Q}$-algebra ${\mathcal A}_{\cst}^{\cyc}$ is called the \textit{cyclic shuffle algebra} of $\cst$, and is also graded by length.

\subsection{Two general results on cyclic shuffle algebras} \label{ss-2rcsa}

We now give two general results on cyclic shuffle algebras, which are analogous to Theorems 3.2 and 3.3 of \cite{Gessel2018} on linear shuffle algebras. We provide proofs for completeness, although they follow in essentially the same way as the proofs of the corresponding results in \cite{Gessel2018}.

Given two cyclic permutation statistics $\cst_{1}$ and $\cst_{2}$, we say that $\cst_{1}$ is a \textit{refinement} of $\cst_{2}$ if for all cyclic permutations $[\pi]$ and $[\sigma]$ of the same length, $\cst_{1}[\pi]=\cst_{1}[\sigma]$ implies $\cst_{2}[\pi]=\cst_{2}[\sigma]$; when this is true, we also say that $\cst_{2}$ is a \textit{coarsening} of $\cst_{1}$. Coarsenings of the cyclic descent set are called \textit{cyclic descent statistics}.
\begin{thm} \label{t-refine}
Suppose that $\cst_{1}$ is cyclic shuffle-compatible and is a refinement of $\cst_{2}$. Let $A$ be a $\mathbb{Q}$-algebra with basis $\{v_{\alpha}\}$ indexed by $\cst_{2}$-equivalence classes $\alpha$, and suppose that there exists a $\mathbb{Q}$-algebra homomorphism $\phi\colon{\mathcal A}_{\cst_{1}}^{\cyc}\rightarrow A$ such that for every $\cst_{1}$-equivalence class $\beta$, we have $\phi(\beta)=v_{\alpha}$ where $\alpha$ is the $\cst_{2}$-equivalence class containing $\beta$. Then $\cst_{2}$ is cyclic shuffle-compatible and the map $v_{\alpha}\mapsto\alpha$ extends by linearity to an isomorphism from $A$ to ${\mathcal A}_{\cst_{2}}^{\cyc}$.
\end{thm}

\begin{proof}
It suffices to show that for any disjoint $\pi$ and $\sigma$, we have
\[
v_{[\pi]_{\cst_{2}}}v_{[\sigma]_{\cst_{2}}}=\sum_{[\tau]\in [\pi]\shu[\sigma]}v_{[\tau]_{\cst_{2}}}.
\]
To that end, we have 
\begin{align*}
v_{[\pi]_{\cst_{2}}}v_{[\sigma]_{\cst_{2}}} & =\phi([\pi]_{\cst_{1}})\phi([\sigma]_{\cst_{1}})\\
 & =\phi([\pi]_{\cst_{1}}[\sigma]_{\cst_{1}})\\
 & =\phi\Bigg(\sum_{[\tau]\in [\pi]\shu[\sigma]}[\tau]_{\cst_{1}}\Bigg)\\
 & =\sum_{[\tau]\in [\pi]\shu[\sigma]}v_{[\tau]_{\cst_{2}}},
\end{align*}
which completes the proof.
\end{proof}

We say that $\cst_{1}$ and $\cst_{2}$ are \textit{equivalent} if $\cst_{1}$ is a simultaneously a refinement and a coarsening of $\cst_{2}$, that is, if for all cyclic permutations $[\pi]$ and $[\sigma]$ of the same length, $\cst_{1} [\pi]=\cst_{1} [\sigma]$ implies $\cst_{2} [\pi]=\cst_{2} [\sigma]$ and vice versa.

\begin{thm} \label{t-equiv}
Let $\cst_{1}$ and $\cst_{2}$ be equivalent cyclic permutation statistics. If $\cst_{1}$ is cyclic shuffle-compatible with cyclic shuffle algebra ${ \mathcal{A}}_{\cst_{1}}^{\cyc}$, then $\cst_{2}$ is also cyclic shuffle-compatible with cyclic shuffle algebra ${ \mathcal{A}}_{\cst_{2}}^{\cyc}$ isomorphic to ${\mathcal{A}}_{\cst_{1}}^{\cyc}$.
\end{thm}

\begin{proof}
Because equivalent statistics have the same equivalence classes on cyclic permutations, we know that ${\mathcal{A}}_{\cst_{1}}^{\cyc}$ and ${\mathcal{A}}_{\cst_{2}}^{\cyc}$ have the same basis elements. Since $\cst_{1}$ and $\cst_{2}$ are equivalent, we have 
\[
[\pi]_{\st_{2}}[\sigma]_{\st_{2}}=[\pi]_{\st_{1}}[\sigma]_{\st_{1}}=\sum_{[\tau]\in[\pi]\shuffle[\sigma]}[\tau]_{\st_{1}}=\sum_{[\tau]\in[\pi]\shuffle[\sigma]}[\tau]_{\st_{2}},
\]
which proves the result.
\end{proof}

\subsection{Symmetries and cyclic shuffle algebras} \label{ss-symmetries}

Many permutation statistics---both linear and cyclic---are related via various symmetries, such as reversal, complementation, and reverse-complementation. For a linear permutation $\pi=\pi_{1}\pi_{2}\cdots\pi_{n}\in\mathfrak{P}_{n}$, we define the \textit{reversal} $\pi^{r}$ of $\pi$ by $\pi^{r}\coloneqq\pi_{n}\pi_{n-1}\cdots\pi_{1}$, the \textit{complement} $\pi^{c}$ of $\pi$ to be the permutation obtained by (simultaneously) replacing the $i$th smallest letter in $\pi$ with the $i$th largest letter in $\pi$ for all $1\leq i\leq n$, and the \textit{reverse-complement} $\pi^{rc}$ of $\pi$ by $\pi^{rc}\coloneqq(\pi^{r})^{c}=(\pi^{c})^{r}$. For example, given $\pi=318269$, we have $\pi^{r}=962813$, $\pi^{c}=692831$, and $\pi^{rc}=138296$.

More generally, let $f$ be an involution on linear permutations which
preserves the length, i.e., $\left|f(\pi)\right|=\left|\pi\right|$
for all $\pi$. We shall write $\pi^{f}$ in place of $f(\pi)$. For
a set $S$ of permutations, let
\[
S^{f}\coloneqq\{\,\pi^{f}:\pi\in S\,\},
\]
so $f$ induces an involution on sets of permutations as well. In
particular, this lets us define $[\pi]^{f}$ for a cyclic permutation
$[\pi]$. Going further, if $C$ is a set of cyclic permutations,
then
\[
C^{f}\coloneqq\{\,[\pi]^{f}:[\pi]\in C\,\}.
\]

Following Gessel and Zhuang \cite{Gessel2018}, we say that $f$ is \textit{shuffle-compatibility-preserving} if for any pair of disjoint permutations $\pi$ and $\sigma$, there exist disjoint permutations $\hat{\pi}$ and $\hat{\sigma}$ with the same relative order as $\pi$ and $\sigma$, respectively, such that $(\pi\shuffle\sigma)^{f}=\hat{\pi}^{f}\shuffle\hat{\sigma}^{f}$ and
$(\hat{\pi}\shuffle\hat{\sigma})^{f}=\pi^{f}\shuffle\sigma^{f}$. (This definition implies that $\pi^{f}$ and $\sigma^{f}$ are disjoint, and similarly with $\hat{\pi}^{f}$ and $\hat{\sigma}^{f}$.) 

Furthermore, we call two linear permutation statistics $\st_{1}$ and $\st_{2}$ $f$-\textit{equivalent} if $\st_{1}\circ f$ is equivalent to $\st_{2}$\textemdash that is, $\st_{1}\pi^{f}=\st_{1}\sigma^{f}$ if and only if $\st_{2}\pi=\st_{2}\sigma$. In other words, $\st_{1}$ and $\st_{2}$ are $f$-equivalent if and only if $(\pi^f)_{\st_{1}}=(\pi_{\st_{2}})^{f}$ for all $\pi$. It is easy to see that, if $\st_{1}\pi^{f}=\st_{2}\pi$ for all $\pi$, then $\st_{1}$ and $\st_{2}$ are $f$-equivalent (although this is not a necessary condition). 

For example, the peak set $\Pk$ is $c$-equivalent to the \textit{valley set} $\Val$ defined in the following way. We call $i\in\{2,3,\dots,n-1\}$ a \textit{valley} of $\pi\in\mathfrak{P}_{n}$ if $\pi_{i-1}>\pi_{i}<\pi_{i+1}$, and we let $\Val\pi$ be the set of valleys of $\pi$. We also define $\val\pi$ to be the number of valleys of $\pi$; then, $\pk$ and $\val$ are $c$-equivalent as well.

Despite its name, $f$-equivalence is not an equivalence relation
(although it is symmetric). However, it turns out that if the statistics involved are shuffle-compatible, then $f$-equivalences induce isomorphisms on the corresponding shuffle algebras. This idea is expressed in the following theorem, which is Theorem 3.5 of Gessel and Zhuang \cite{Gessel2018}.

\begin{thm} \label{t-linsym}
Let $f$ be shuffle-compatibility-preserving, and suppose that $\st_{1}$ and $\st_{2}$ are $f$-equivalent \textup{(}linear\textup{)} permutation statistics. If $\st_{1}$ is shuffle-compatible with shuffle algebra $\mathcal{A}_{\st_{1}}$, then $\st_{2}$ is also shuffle-compatible, and the linear map defined by $\pi_{\st_{1}}\mapsto\pi_{\st_{2}}^{f}$ is a $\mathbb{Q}$-algebra isomorphism between their shuffle algebras $\mathcal{A}_{\st_{1}}$ and $\mathcal{A}_{\st_{2}}$.
\end{thm}

Gessel and Zhuang proved that reversal, complementation, and reverse-complementation are all shuffle-compatibility-preserving. Thus, they were able to use Theorem \ref{t-linsym} to prove a collection of shuffle-compatibility results for statistics that are $r$-, $c$-, or $rc$-equivalent to another statistic whose shuffle-compatibility had already been established. For example, it follows from the shuffle-compatibility of the peak set $\Pk$ that the valley set $\Val$ is shuffle-compatible with shuffle algebra $\mathcal{A}_{\Val}$ isomorphic to $\mathcal{A}_{\Pk}$.

Moving onto the cyclic setting, let us call $f$ \textit{rotation-preserving} if $[\pi]^{f}=[\pi^{f}]$ for all $\pi$. We now prove that if $f$ is both shuffle-compatibility-preserving and rotation-preserving, then $f$ satisfies a cyclic version of the shuffle-compatibility-preserving property.
\begin{lem} \label{l-cscp}
If $f$ is shuffle-compatibility-preserving and rotation-preserving, then for any pair of disjoint permutations $\pi$ and $\sigma$, there exist disjoint permutations $\hat{\pi}$ and $\hat{\sigma}$ with the same relative order as $\pi$ and $\sigma$, respectively, for which $([\pi]\shuffle[\sigma])^{f}=[\hat{\pi}^{f}]\shuffle[\hat{\sigma}^{f}]$ and $([\hat{\pi}]\shuffle[\hat{\sigma}])^{f}=[\pi^{f}]\shuffle[\sigma^{f}]$.
\end{lem}

\begin{proof}
Let $[\tau]\in[\pi]\shuffle[\sigma]$, so that $\tau\in\bar{\pi}\shuffle\bar{\sigma}$ for some $\bar{\pi}\in[\pi]$ and $\bar{\sigma}\in[\sigma]$, and thus $\tau^{f}\in(\bar{\pi}\shuffle\bar{\sigma})^{f}$. Since $f$ is shuffle-compatibility-preserving, we have that $\tau^{f}\in\hat{\bar{\pi}}^{f}\shuffle\hat{\bar{\sigma}}^{f}$ where $\hat{\bar{\pi}}$ and $\hat{\bar{\sigma}}$ are disjoint permutations with the same relative order as $\bar{\pi}$ and $\bar{\sigma}$, respectively. Since $\bar{\pi}$ is a rotation of $\pi$ and $\hat{\bar{\pi}}$ has the same relative order as $\bar{\pi}$, it follows that $\hat{\bar{\pi}}$ is a rotation of a permutation $\hat{\pi}$ with the same relative order as $\pi$, and similarly $\hat{\bar{\sigma}}$ is a rotation of a permutation $\hat{\sigma}$ with the same relative order as $\sigma$.  Clearly, $\hat{\pi}$ and $\hat{\sigma}$ are disjoint because $\hat{\bar{\pi}}$ and $\hat{\bar{\sigma}}$ are disjoint. Because $f$ is rotation-preserving, $\hat{\bar{\pi}}\in[\hat{\pi}]$ and $\hat{\bar{\sigma}}\in[\hat{\sigma}]$ imply $\hat{\bar{\pi}}^{f}\in[\hat{\pi}^{f}]$ and $\hat{\bar{\sigma}}\in[\hat{\sigma}^{f}]$. Therefore, $\tau^{f}\in\hat{\bar{\pi}}^{f}\shuffle\hat{\bar{\sigma}}^{f}$ implies $[\tau]^{f}=[\tau^{f}]\in[\hat{\pi}^{f}]\shuffle[\hat{\sigma}^{f}]$. 

We have shown that $([\pi]\shuffle[\sigma])^{f}$ is a
subset of $[\hat{\pi}^{f}]\shuffle[\hat{\sigma}^{f}]$, but since these two sets have the same cardinality, they are in fact equal. 
We omit the proof of $([\hat{\pi}]\shuffle[\hat{\sigma}])^{f}=[\pi^{f}]\shuffle[\sigma^{f}]$ as it is similar.
\end{proof}

\begin{lem} \label{l-rotp}
Reversal, complementation, and reverse-complementation are all rotation-preserving.
\end{lem}

\begin{proof}
Let $\pi=\pi_{1}\pi_{2}\cdots\pi_{n}$ be a (linear) permutation. We have 
\begin{align*}
[\pi]^{r} & =\{\pi_{1}\pi_{2}\cdots\pi_{n},\pi_{n}\pi_{1}\cdots\pi_{n-1},\dots,\pi_{2}\cdots\pi_{n}\pi_{1}\}^{r}\\
 & =\{\pi_{n}\cdots\pi_{2}\pi_{1},\pi_{n-1}\cdots\pi_{1}\pi_{n},\dots,\pi_{1}\pi_{n}\cdots\pi_{2}\}\\
 & =[\pi^{r}],
\end{align*}
so reversal is rotation-preserving. Moreover, it is clear that taking the complement of the permutation $\pi_{i+1}\cdots\pi_{n}\pi_{1}\cdots\pi_{i}$ (obtained by rotating the last $n-i$ letters of $\pi$ to the front) yields the same result as first taking the complement of $\pi$ and then rotating the last $n-i$ letters of $\pi^{c}$ to the front, so complementation is rotation-preserving. Lastly, since we have established that $[\pi^{c}]=[\pi]^{c}$ for all permutations $\pi$, we can replace $\pi$ by $\pi^{r}$ to obtain $[\pi^{rc}]=[\pi^{r}]^{c}=[\pi]^{rc}$, so reverse-complementation is rotation-preserving as well.
\end{proof}

In analogy with $f$-equivalence of linear permutation statistics,
let us call two cyclic permutation statistics $\cst_{1}$ and $\cst_{2}$
\textit{$f$-equivalent} if $\cst_{1}\circ f$ is equivalent to $\cst_{2}$,
or equivalently, if $[\pi^{f}]_{\cst_{1}}=([\pi]_{\cst_{2}})^{f}$.
The following is a cyclic version of Theorem \ref{t-linsym}.

\begin{thm}
\label{t-cycsym}Let $f$ be shuffle-compatibility-preserving and
rotation-preserving, and let $\cst_{1}$ and $\cst_{2}$ be $f$-equivalent
cyclic permutation statistics. If $\cst_{1}$ is cyclic shuffle-compatible,
then $\cst_{2}$ is cyclic shuffle-compatible with $\mathcal{A}_{\cst_{2}}^{\cyc}$
isomorphic to $\mathcal{A}_{\cst_{1}}^{\cyc}$.
\end{thm}

\begin{proof}
Let $[\pi]$ and $[\tilde{\pi}]$ be cyclic permutations in the same $\cst_{2}$-equivalence class, and similarly with $[\sigma]$ and $[\tilde{\sigma}]$, such that $\pi$ and $\sigma$ are disjoint and $\tilde{\pi}$ and $\tilde{\sigma}$ are disjoint. We know from 
Lemma \ref{l-cscp} that there exist permutations $\hat{\pi}$, $\hat{\sigma}$, $\hat{\tilde{\pi}}$, and $\hat{\tilde{\sigma}}$\textemdash having the same relative order as $\pi$, $\sigma$, $\tilde{\pi}$, and $\tilde{\sigma}$, respectively\textemdash satisfying $([\pi]\shuffle[\sigma])^{f}=[\hat{\pi}^{f}]\shuffle[\hat{\sigma}^{f}]$, 
$([\hat{\pi}]\shuffle[\hat{\sigma}])^{f}
=[\pi^{f}]\shuffle[\sigma^{f}]$, $([\tilde{\pi}]\shuffle[\tilde{\sigma}])^{f}=[\hat{\tilde{\pi}}^{f}]\shuffle[\hat{\tilde{\sigma}}^{f}]$, and $([\hat{\tilde{\pi}}]\shuffle[\hat{\tilde{\sigma}}])^{f}=[\tilde{\pi}^{f}]\shuffle[\tilde{\sigma}^{f}]$.

Because $\hat{\pi}$ and $\hat{\tilde{\pi}}$ have the same relative order as $\pi$ and $\tilde{\pi}$, respectively, we have
\[
[\hat{\pi}]_{\cst_{2}}=[\pi]_{\cst_{2}}=[\tilde{\pi}]_{\cst_{2}}=[\hat{\tilde{\pi}}]_{\cst_{2}}.
\]
Then, because $\cst_{1}$ and $\cst_{2}$ are $f$-equivalent, we have 
\[
[\hat{\pi}^{f}]_{\cst_{1}}=([\hat{\pi}]_{\cst_{2}})^{f}=([\hat{\tilde{\pi}}]_{\cst_{2}})^{f}=[\hat{\tilde{\pi}}^{f}]_{\cst_{1}},
\]
so $[\hat{\pi}^{f}]$ and $[\hat{\tilde{\pi}}^{f}]$ are $\cst_{1}$-equivalent. The same reasoning shows that $[\hat{\sigma}^{f}]$ and $[\hat{\tilde{\sigma}}^{f}]$ are also $\cst_{1}$-equivalent.

By cyclic shuffle-compatibility of $\cst_{1}$, we have the multiset equality 
\[
\{\{\,\cst_{1}[\tau]:[\tau]\in[\hat{\pi}^{f}]\shuffle[\hat{\sigma}^{f}]\,\}\}=\{\{\,\cst_{1}[\tau]:[\tau]\in[\hat{\tilde{\pi}}^{f}]\shuffle[\hat{\tilde{\sigma}}^{f}]\,\}\},
\]
which---by $f$-equivalence of $\cst_{1}$ and $\cst_{2}$---is equivalent to
\[
\{\{\,\cst_{2}[\tau^{f}]:[\tau]\in[\hat{\pi}^{f}]\shuffle[\hat{\sigma}^{f}]\,\}\}=\{\{\,\cst_{2}[\tau^{f}]:[\tau]\in[\hat{\tilde{\pi}}^{f}]\shuffle[\hat{\tilde{\sigma}}^{f}]\,\}\},
\]
which is in turn equivalent to
\[
\{\{\,\cst_{2}[\tau]:[\tau]^{f}\in[\hat{\pi}^{f}]\shuffle[\hat{\sigma}^{f}]\,\}\}=\{\{\,\cst_{2}[\tau]:[\tau]^{f}\in[\hat{\tilde{\pi}}^{f}]\shuffle[\hat{\tilde{\sigma}}^{f}]\,\}\}
\]
because $f$ is rotation-preserving. Since $([\pi]\shuffle[\sigma])^{f}=[\hat{\pi}^{f}]\shuffle[\hat{\sigma}^{f}]$ and $([\tilde{\pi}]\shuffle[\tilde{\sigma}])^{f}=[\hat{\tilde{\pi}}^{f}]\shuffle[\hat{\tilde{\sigma}}^{f}]$, we have
\[
\{\{\,\cst_{2}[\tau]:[\tau]\in[\pi]\shuffle[\sigma]\,\}\}=\{\{\,\cst_{2}[\tau]:[\tau]\in[\tilde{\pi}]\shuffle[\tilde{\sigma}]\,\}\},
\]
which shows that $\cst_{2}$ is cyclic shuffle-compatible.

It remains to prove that $\mathcal{A}_{\cst_{2}}^{\cyc}$ is isomorphic to $\mathcal{A}_{\cst_{1}}^{\cyc}$. Define the linear map $\lambda\colon\mathcal{A}_{\cst_{2}}^{\cyc}\rightarrow\mathcal{A}_{\cst_{1}}^{\cyc}$
by $[\pi]_{\cst_{2}}\mapsto[\pi^{f}]_{\cst_{1}}$. Observe that 
\[
\sum_{[\tau]\in[\pi]\shuffle[\sigma]}[\tau]_{\cst_{2}}=\sum_{[\tau]\in[\hat{\pi}]\shuffle[\hat{\sigma}]}[\tau]_{\cst_{2}}
\]
because $\cst_{2}$ is cyclic shuffle-compatible, and thus we have
\begin{align*}
\lambda([\pi]_{\cst_{2}}[\sigma]_{\cst_{2}}) & =\lambda\Big(\sum_{[\tau]\in[\pi]\shuffle[\sigma]}[\tau]_{\cst_{2}}\Big)\\
 & =\lambda\Big(\sum_{[\tau]\in[\hat{\pi}]\shuffle[\hat{\sigma}]}[\tau]_{\cst_{2}}\Big)\\
 & =\sum_{[\tau]\in[\hat{\pi}]\shuffle[\hat{\sigma}]}[\tau^{f}]_{\cst_{1}}\\
 & =\sum_{[\tau]^{f}\in[\hat{\pi}]\shuffle[\hat{\sigma}]}[\tau]_{\cst_{1}}\\
 & =\sum_{[\tau]\in[\pi^{f}]\shuffle[\sigma^{f}]}[\tau]_{\cst_{1}}\\
 & =[\pi^{f}]_{\cst_{1}}[\sigma^{f}]_{\cst_{1}}\\
 & =\lambda([\pi]_{\cst_{2}})\lambda([\sigma]_{\cst_{2}}).
\end{align*}
Hence, $\lambda$ is a $\mathbb{Q}$-algebra isomorphism 
from $\mathcal{A}_{\cst_{2}}^{\cyc}$ to $\mathcal{A}_{\cst_{1}}^{\cyc}$.
\end{proof}

\begin{cor}
Suppose that the cyclic permutation statistics $\cst_{1}$ and $\cst_{2}$ are $r$-equivalent, $c$-equivalent, or $rc$-equivalent. If $\cst_{1}$ is cyclic shuffle-compatible, then $\cst_{2}$ is cyclic shuffle-compatible with cyclic shuffle algebra $\mathcal{A}_{\cst_{2}}^{\cyc}$ isomorphic to $\mathcal{A}_{\cst_{1}}^{\cyc}$.
\end{cor}

\subsection{Constructing cyclic shuffle algebras from linear ones}

The following theorem---one of the main results of this paper---allows us to construct cyclic shuffle algebras from shuffle algebras of shuffle-compatible (linear) permutation statistics.

\begin{thm} \label{t-AtoAcyc}
Let $\cst$ be a cyclic permutation statistic and let $\st$ be a shuffle-compatible \textup{(}linear\textup{)} permutation statistic. Given a cyclic permutation $[\pi]$, let 
\[
v_{[\pi]}=\sum_{\bar{\pi}\in[\pi]}\bar{\pi}_{\st}\in{\mathcal A}_{\st}.
\]
Suppose that $v_{[\pi]}=v_{[\sigma]}$ whenever $[\pi]$ and $[\sigma]$ are $\cst$-equivalent, and that $\{v_{[\pi]}\}$ \textup{(}ranging over all $\cst$-equivalence classes\textup{)} is linearly independent. Then $\cst$ is cyclic shuffle-compatible and the map $\psi_{\cst}\colon{\mathcal A}_{\cst}^{\cyc}\rightarrow{\mathcal A}_{\st}$ given by 
$$
\psi_{\cst}([\pi]_{\cst})=v_{[\pi]}
$$
extends linearly to a $\mathbb{Q}$-algebra isomorphism from ${\mathcal A}_{\cst}^{\cyc}$ to the span of $\{v_{[\pi]}\}$, a subalgebra of ${\mathcal A}_{\st}$.
\end{thm}

\begin{proof}
Since $v_{[\pi]}=v_{[\sigma]}$ whenever $[\pi]$ and $[\sigma]$ are $\cst$-equivalent, we know that $\psi_{\cst}$ is a well-defined linear map on ${\mathcal A}_{\cst}^{\cyc}$. (We do not yet know whether ${\mathcal A}_{\cst}^{\cyc}$ is an algebra; here we are only considering ${\mathcal A}_{\cst}^{\cyc}$ as a vector space.) Furthermore, because $\{v_{[\pi]}\}$ is linearly independent, the linear map $\psi_{\cst}$ is a vector space isomorphism from ${\mathcal A}_{\cst}^{\cyc}$ to a subspace of ${\mathcal A}_{\st}$. 

To show that $\cst$ is cyclic shuffle-compatible, we show that 
\[
[\pi]_{\cst}[\sigma]_{\cst}=\sum_{[\tau]\in [\pi]\shu[\sigma]}[\tau]_{\cst}
\]
is a well-defined multiplication in ${\mathcal A}_{\cst}^{\cyc}$. Let $[\pi^{\prime}],[\pi^{\prime\prime}]\in[\pi]{}_{\cst}$ and let $[\sigma^{\prime}],[\sigma^{\prime\prime}]\in[\sigma]_{\cst}$, where $\pi^{\prime}$ and $\sigma^{\prime}$ are disjoint and so are $\pi^{\prime\prime}$ and $\sigma^{\prime\prime}$. Then 
\begin{align*}
\psi_{\cst}\Bigg(\sum_{[\tau]\in[\pi^{\prime}]\shu[\sigma^{\prime}]}[\tau]_{\cst}\Bigg) & =\sum_{[\tau]\in[\pi^{\prime}]\shuffle[\sigma^{\prime}]}v_{[\tau]}\\
 & =\sum_{[\tau]\in[\pi^{\prime}]\shu[\sigma^{\prime}]}\sum_{\bar{\tau}\in[\tau]}\bar{\tau}_{\st}\\
 & =\sum_{\bar{\pi}\in[\pi^{\prime}]}\sum_{\bar{\sigma}\in[\sigma^{\prime}]}\sum_{\bar{\tau}\in\bar{\pi}\shu\bar{\sigma}}\bar{\tau}_{\st}\\
 & =v_{[\pi^{\prime}]}v_{[\sigma^{\prime}]}
\end{align*}
and similarly 
\[
\psi_{\cst}\Bigg(\sum_{[\tau]\in [\pi^{\prime\prime}]\shu[\sigma^{\prime\prime}]}[\tau]_{\cst}\Bigg)=v_{[\pi^{\prime\prime}]}v_{[\sigma^{\prime\prime}]}.
\]
Since $[\pi^\prime]$ and $[\pi^{\prime\prime}]$ are $\cst$-equivalent and similarly with $[\sigma^\prime]$ and $[\sigma^{\prime\prime}]$, we have 
\[
\psi_{\cst}\Bigg(\sum_{[\tau]\in [\pi^{\prime}]\shu[\sigma^{\prime}]}[\tau]_{\cst}\Bigg)=v_{[\pi^\prime]}v_{[\sigma^\prime]}=v_{[\pi^{\prime\prime}]}v_{[\sigma^{\prime\prime}]}=\psi_{\cst}\Bigg(\sum_{[\tau]\in [\pi^{\prime\prime}]\shu[\sigma^{\prime\prime}]}[\tau]_{\cst}\Bigg)
\]
and thus 
\[
\sum_{[\tau]\in [\pi^{\prime}]\shu[\sigma^{\prime}]}[\tau]_{\cst}=\sum_{[\tau]\in [\pi^{\prime\prime}]\shu[\sigma^{\prime\prime}]}[\tau]_{\cst}
\]
due to injectivity of $\psi_{\cst}$. We have shown that the multiplication of the cyclic shuffle algebra ${\mathcal A}_{\cst}^{\cyc}$ is well-defined, and therefore $\cst$ is shuffle-compatible. 

Finally, we have
\allowdisplaybreaks
\begin{align*}
\psi_{\cst}([\pi]_{\cst}[\sigma]_{\cst}) & =\psi_{\cst}\Bigg(\sum_{[\tau]\in [\pi]\shu[\sigma]}[\tau]_{\cst}\Bigg)\\
 & =v_{[\pi]}v_{[\sigma]}\\
 & =\psi_{\cst}([\pi]_{\cst})\psi_{\cst}([\sigma]_{\cst}),
\end{align*}
%\begin{equation*}
%\psi_{\cst}([\pi]_{\cst}[\sigma]_{\cst})  =\psi_{\cst}\Bigg(\sum_{[\tau]\in [\pi]\shu[\sigma]}%[\tau]_{\cst}\Bigg)
%  =v_{[\pi]}v_{[\sigma]}
%  =\psi_{\cst}([\pi]_{\cst})\psi_{\cst}([\sigma]_{\cst})
%\end{equation*}
so $\psi_{\cst}$ is a $\mathbb{Q}$-algebra isomorphism from ${\mathcal A}_{\cst}^{\cyc}$ to the span of $\{v_{[\pi]}\}$.
\end{proof}

\section{Shuffle-compatibility and quasisymmetric functions} \label{s-scqsym}

The focus of this section is the relationship between cyclic shuffle-compatibility and cyclic quasisymmetric functions. We shall begin by providing the necessary background on descent compositions, cyclic descent compositions, and (ordinary) quasisymmetric functions.

\subsection{Descent compositions}

Every permutation can be uniquely decomposed into a sequence of maximal increasing consecutive subsequences, which we call \textit{increasing runs}. For example, the increasing runs
of $4783291$ are $478$, $3$, $29$, and $1$. Equivalently, an increasing run of $\pi$ is a maximal consecutive subsequence with no descents.

The number of increasing runs of a nonempty permutation is one more than its number of descents; in fact, the lengths of the increasing runs determine the descents, and vice versa. Given a subset $S\subseteq[n-1]$ with elements $s_{1}<s_{2}<\cdots<s_{j}$, let $\Comp S$ be the composition of $n$ defined by
$$\Comp S \coloneqq (s_{1},s_{2}-s_{1},\dots,s_{j}-s_{j-1},n-s_{j});$$
also, given a composition  $L=(L_{1},L_{2},\dots,L_{k})$, let 
$$\Des L \coloneqq\{L_{1},L_{1}+L_{2},\dots,L_{1}+\cdots+L_{k-1}\}$$ 
be the corresponding subset of $[n-1]$.
It is straightforward to verify that $\Comp$ and $\Des$ are inverse bijections. If $\pi \in \mathfrak{P}_n$ has descent set $S\subseteq[n-1]$, then we say that $\Comp S$ is the \textit{descent composition} of $\pi$, which we also denote by $\Comp \pi$. By convention, the empty permutation has descent composition $\varnothing$. 

Continuing the example above, we have $\Comp 4783291 = (3,1,2,1)$. Observe that the descent composition of $\pi$ gives the lengths of the increasing runs of $\pi$ in the order that they appear. Conversely, if $\pi$ has descent composition $L$, then its descent set $\Des \pi$ is $\Des L$.

We call a permutation statistic $\st$ a \textit{descent statistic} if it depends only on the descent composition, that is, if $\Comp \pi=\Comp \sigma$ implies $\st \pi=\st \sigma$. Equivalently, a descent statistic depends only on the descent set and length. If $\st$ is a descent statistic, then we can extend the notion of $\st$-equivalence classes of permutations to that of compositions. First, let $\st L$ indicate the value of $\st$ on any permutation with descent composition $L$. Then we say that two compositions $L$ and $K$ of the same size---where the {\em size} of a composition is the sum of its parts---are $\st$\textit{-equivalent} if $\st L$ = $\st K$. For example, the compositions $(2,3,1)$ and $(1,1,4)$ are $\des$-equivalent because any permutation with one of these descent compositions has exactly two descents.

\subsection{Cyclic descent compositions}

The notion of descent compositions for linear permutations can be extended to cyclic permutations. To do so, we shall need a few more preliminary definitions. A \textit{cyclic shift} of a composition $L=(L_1,L_2,\dots,L_k)$ is a composition of the form 
$$(L_j, L_{j+1}, \dots, L_k, L_1, \dots, L_{j-1}).$$
A \textit{cyclic composition} of $n$ is then the equivalence class of a composition of $n$ under cyclic shift. For example, 
$$
[2,1,3] = \{ (2,1,3), (1,3,2), (3,2,1) \}
$$
and 
$$
[1,2,1,2] = \{ (1,2,1,2), (2,1,2,1) \}
$$
are both cyclic compositions. By convention, we'll also allow the empty set $\varnothing$ to be a cyclic composition.

Let us call $S$ a \textit{non-Escher}\footnote{We borrow the term ``non-Escher'' from \cite{Adin2021} and other recent works on cyclic descent extensions. As explained there, this term is a reference to M.\ C.\ Escher's painting ``Ascending and Descending''.} subset of $[n]$ if $S$ is the cyclic descent set of some linear permutation of length $n$. When $n=0$ or $n=1$, only the empty set is non-Escher, and when $n\geq 2$, all subsets of $[n]$ are non-Escher except for the empty set and $[n]$ itself. We associate to each non-Escher subset $S\subseteq [n]$ a composition $\cComp S$ defined by 
\[
\cComp S\coloneqq
\begin{cases}
(s_{2}-s_{1},\dots,s_{j}-s_{j-1},n-s_{j}+s_{1}), & \text{if }n\geq2,\\
(1), & \text{if }n=1,\\
\emptyset, & \text{if }n=0.
\end{cases}
\]
It is easy to see that if $S^\prime$ is a cyclic shift of $S$, then $\cComp S^\prime$ is a cyclic shift of $\cComp S$. So, if $[S]$ is the equivalence class of $S$ under cyclic shift, then we can let $\cComp[S]$ be the cyclic composition defined by
$$\cComp[S] \coloneqq [\cComp S].$$
We say that a cyclic composition is \textit{non-Escher} if it is an image of this induced map $\cComp$, and one can check that $\cComp$ is a bijection from equivalence classes of non-Escher subsets of $[n]$ under cyclic shift to non-Escher cyclic compositions of $n$. If $S$ is the cyclic descent set of a linear permutation $\pi$, then we call $\cComp[S]$ the \textit{cyclic descent composition} of the cyclic permutation $[\pi]$. We denote the cyclic descent composition of $[\pi]$ simply as $\cComp [\pi]$.

For example, take $\pi = 179624$. Then $\pi$ has cyclic descent set $S = \{3,4,6\}$, so the cyclic descent composition of $[\pi]$ is $\cComp[S] = [1,2,3]$, which we also denote by $\cComp[\pi]$.

A cyclic permutation statistic $\cst$ is called a \textit{cyclic descent statistic} if it depends only on the cyclic descent composition---that is, if $\cComp [\pi] = \cComp [\sigma]$ implies $\cst [\pi] = \cst [\sigma]$. (This is equivalent to the definition given in Section \ref{ss-2rcsa}.) Similar to the notation $\st L$, we can write $\cst [L]$ for the value of $\cst$ on any cyclic permutation with cyclic descent composition $[L]$, and we shall say that two cyclic compositions $[L]$ and $[K]$ of the same size---which means that $L$ and $K$ have the same size---are $\cst$\textit{-equivalent} if $\cst [L] = \cst [K]$.

\subsection{Quasisymmetric functions}

A formal power series $f\in\mathbb{Q}[[x_{1},x_{2},\dots]]$ of bounded degree in countably many commuting variables $x_{1},x_{2},\dots$ is called a \textit{quasisymmetric function} if for any positive integers $a_{1},a_{2},\dots,a_{k}$, $i_{1}<i_{2}<\cdots<i_{k}$, and $j_{1}<j_{2}<\cdots<j_{k}$, we have equality of the monomial coefficients
\[
[x_{i_{1}}^{a_{1}}x_{i_{2}}^{a_{2}}\cdots x_{i_{k}}^{a_{k}}]\,f=[x_{j_{1}}^{a_{1}}x_{j_{2}}^{a_{2}}\cdots x_{j_{k}}^{a_{k}}]\,f.
\]

The $\mathbb{Q}$-vector space $\QSym_{n}$ of quasisymmetric functions homogeneous of degree $n$ has dimension $2^{n-1}$, the number of compositions of $n$. An important basis of $\QSym_{n}$ is the basis of \textit{fundamental quasisymmetric functions} $\{F_{n,L}\}_{L\vDash n}$ defined by
\[
F_{n,L}\coloneqq\sum_{\substack{i_{1}\leq i_{2}\leq\cdots\leq i_{n}\\
i_{j}<i_{j+1}\,\mathrm{if}\,j\in\Des L
}
}x_{i_{1}}x_{i_{2}}\cdots x_{i_{n}}.
\]
Sometimes, it is more convenient to index fundamental quasisymmetric functions by subsets of $[n-1]$ as opposed to compositions of $n$, in which case we'll use the notation 
$$F_{n,S} \coloneqq F_{n,\Comp S}.$$

The product of two quasisymmetric functions is again quasisymmetric. The multiplication rule for the fundamental basis is given by the following theorem, which can be proved using $P$-partitions; see \cite[Exercise 7.93]{Stanley2001}.
\begin{thm}
\label{t-fqsym} 
%Let $c_{J,K}^{L}$ be the number of permutations with descent composition $L$ among the shuffles of two disjoint permutations, one with descent composition $J$ and the other with descent composition $K$. Then 
%\begin{equation}
%F_{J}F_{K}=\sum_{L}c_{J,K}^{L}F_{L}.\label{e-fundshuffle}
%\end{equation}
Let $m$ and $n$ be non-negative integers, and let $A\subseteq[m-1]$ and $B\subseteq[n-1]$. Then
\begin{equation*}
F_{m,A}F_{n,B}=\sum_{\tau\in\pi\shuffle\sigma}F_{m+n,\Des\tau}
\end{equation*}
where $\pi$ is any permutation of length $m$ with descent set $A$ and $\sigma$ is any permutation \textup{(}disjoint from $\pi$\textup{)} of length $n$ with descent set $B$.
\end{thm}

If $f\in\QSym_{m}$ and $g\in\QSym_{n}$, then $fg\in\QSym_{m+n}$. Therefore $\QSym\coloneqq\bigoplus_{n=0}^{\infty}\QSym_{n}$ is a graded $\mathbb{Q}$-algebra called the \textit{algebra of quasisymmetric functions} (with coefficients in $\mathbb{Q}$), a subalgebra of $\mathbb{Q}[[x_{1},x_{2},\dots]]$. Motivated by Stanley's theory of $P$-partitions, Gessel introduced quasisymmetric functions in \cite{Gessel1984} and developed the basic algebraic properties of $\QSym$. Further properties of $\QSym$ and its connections with many topics of study in combinatorics and algebra were developed in the subsequent decades; see \cite[Section 5]{Grinberg2020},
\cite{Luoto2013},
\cite[Chapter 8]{sag:aoc}, and \cite[Section 7.19]{Stanley2011}  for several basic references.

From Theorem \ref{t-fqsym}, we see that the descent set shuffle algebra ${\mathcal A}_{\Des}$ is isomorphic to $\QSym$; this is Corollary 4.2 of \cite{Gessel2018}.

\subsection{Cyclic quasisymmetric functions and the cyclic shuffle algebra of \texorpdfstring{$\cDes$}{cDes}} 

We are now ready to discuss cyclic quasisymmetric functions and their role in cyclic shuffle-compatibility.

Given a subset $S$ of $[n]$ where $n\geq1$, let 
\[
F_{n,S}^{\cyc}\coloneqq\sum_{i\in[n]}F_{n,(S+i)\cap[n-1]},
\]
and let $F_{0,\emptyset}^{\cyc}\coloneqq1$; these are the \textit{fundamental cyclic quasisymmetric functions} introduced by Adin, Gessel, Reiner, and Roichman \cite{Adin2021}. It is clear from this definition that the $F_{n,S}^{\cyc}$ are invariant under cyclic shift; in other words, if $S^{\prime}=S+i$ for some integer $i$, then $F_{n,S}^{\cyc}=F_{n,S^{\prime}}^{\cyc}$.
As such, if $[S]$ is the equivalence class of the set $S$ under cyclic shift, then it makes sense to define 
$$F_{n,[S]}^{\cyc}\coloneqq F_{n,S}^{\cyc}.$$ 
We can also index fundamental cyclic quasisymmetric functions using compositions; for a composition $L$ of $n$, let
$$F_{L}^{\cyc}\coloneqq F_{n,\cDes L}^{\cyc} \quad\text{and}\quad F_{[L]}^{\cyc}\coloneqq F_{L}^{\cyc}.$$
Note that $n$ is not needed in the subscript when using $L$ or $[L]$ since it is determined from the sum of the parts of $L$.

Let $\cQSym^{-}$ denote the span of $\{F_{n,[S]}^{\cyc}\}$ over all $n\geq0$ and all equivalence classes $[S]$ of non-Escher subsets $S\subseteq[n]$. The following theorem, proven by Adin et al.~\cite[Theorem 3.22]{Adin2021}, gives a multiplication rule for the fundamental cyclic quasisymmetric functions in $\cQSym^{-}$, which also implies that the cyclic descent set $\cDes$ is cyclic shuffle-compatible and has cyclic shuffle algebra isomorphic to $\cQSym^{-}$.

\begin{thm}
\label{t-fcycmult} 
Let $m$ and $n$ be non-negative integers, and let $A\subseteq[m]$ and $B\subseteq[n]$ be non-Escher subsets. Then
\begin{equation}
F_{m,[A]}^{\cyc}F_{n,[B]}^{\cyc}=\sum_{[\tau]\in[\pi]\shuffle[\sigma]}F_{m+n,\cDes[\tau]}^{\cyc}\label{e-fcycmult}
\end{equation}
where $[\pi]$ is any cyclic permutation of length $m$ with cyclic descent set $[A]$ and $[\sigma]$ is any cyclic permutation \textup{(}with $\sigma$ disjoint from $\pi$\textup{)} of length $n$ with cyclic descent set $[B]$.
%Let $m$ and $n$ be non-negative integers, let $[J]$ be a non-Escher cyclic composition of $m$, and let $[K]$ be a non-Escher cyclic composition of $n$. Then 
%\begin{equation}
%F_{[J]}^{\cyc}F_{[K]}^{\cyc}=\sum_{[\tau]\in[\pi]\shuffle[\sigma]}F_{\cComp[\tau]}^{\cyc}\label{e-fcycmult}
%\end{equation}
%where $[\pi]$ is any cyclic permutation of length $m$ with cyclic descent composition $[J]$ and $[\sigma]$ is any cyclic permutation \textup{(}with $\sigma$ disjoint from $\pi$\textup{)} of length $n$ with cyclic descent composition $[K]$.
\end{thm}

Adin et al.\ proved Theorem \ref{t-fcycmult} using toric $[\vec{D}]$-partitions; we now supply an alternative proof using Theorem \ref{t-AtoAcyc}.

\begin{proof}
We know that the descent set ${\mathcal \Des}$ is shuffle-compatible and its shuffle algebra ${\mathcal A}_{\Des}$ is isomorphic to the algebra of quasisymmetric functions, $\QSym$, through the isomorphism $\phi_{\Des}(\pi{}_{\Des})=F_{\left|\pi\right|,\Des(\pi)}$.
Then, using the notation of Theorem \ref{t-AtoAcyc}, we have
\[
\phi_{\Des}(v_{[\pi]})=\phi_{\Des}\Big(\sum_{\bar{\pi}\in[\pi]}\bar{\pi}_{\Des}\Big)=\sum_{i\in[n]}F_{n,(\cDes\pi+i)\cap[n-1]}=F_{n,\cDes[\pi]}^{\cyc}
\]
where $n=\left|\pi\right|$. If $[\pi]$ and $[\sigma]$ are $\cDes$-equivalent, then both $\phi_{\Des}(v_{[\pi]})$ and $\phi_{\Des}(v_{[\sigma]})$ are equal to $F_{n,[S]}^{\cyc}$ where $n=\left|\pi\right|=\left|\sigma\right|$ and $[S]=\cDes[\pi]=\cDes[\sigma]$, so $v_{[\pi]}=v_{[\sigma]}$. The linear independence of the $F_{n,[S]}^{\cyc}$ can be established by showing that the monomial cyclic quasisymmetric functions are linearly independent and expressing each $F_{n,[S]}^{\cyc}$ in terms of monomial cyclic quasisymmetric functions; see \cite[Section 2]{Adin2021} for details. Theorem~\ref{t-AtoAcyc} implies that $\cDes$ is cyclic shuffle-compatible and that $\mathcal{A}_{\cDes}^{\cyc}$ is isomorphic to $\cQSym^{-}$ via the isomorphism $[\pi]_{\cDes}\mapsto F_{\left|\pi\right|,\cDes[\pi]}^{\cyc}$, from which the multiplication rule (\ref{e-fcycmult}) follows.
%We know that the descent set ${\mathcal \Des}$ is shuffle-compatible and its shuffle algebra ${\mathcal A}_{\Des}$ is isomorphic to the algebra of quasisymmetric functions, $\QSym$, via the isomorphism $\phi_{\Des}(\pi{}_{\Des})=F_{\Comp\pi}=F_{\left|\pi\right|,\Des\pi}$.
%Then, using the notation of Theorem \ref{t-AtoAcyc}, we have
%\[
%\phi_{\Des}(v_{[\pi]})=\phi_{\Des}\Big(\sum_{\bar{\pi}\in[\pi]}\bar{\pi}_{\Des}\Big)=\sum_{i\in[n]}F_{n,(\cDes(\pi)+i)\cap[n-1]}=F_{n,\cDes[\pi]}^{\cyc}=F_{\cComp[\pi]}^{\cyc}
%\]
%where $n=\left|\pi\right|$. If $[\pi]$ and $[\sigma]$ are $\cDes$-equivalent, then both $\phi_{\Des}(v_{[\pi]})$ and $\phi_{\Des}(v_{[\sigma]})$ are equal to $F_{[L]}^{\cyc}$ where $[L]=\cComp[\pi]=\cComp[\sigma]$, so $v_{[\pi]}=v_{[\sigma]}$. The linear independence of the $F_{[L]}^{\cyc}$ can be established by showing that the monomial cyclic quasisymmetric functions are linearly independent and expressing each $F_{[L]}^{\cyc}$ in terms of monomial cyclic quasisymmetric functions; see \cite[Section 2]{Adin2021} for details. Theorem \ref{t-AtoAcyc} implies that $\cDes$ is cyclic shuffle-compatible and that $\mathcal{A}_{\cDes}^{\cyc}$ is isomorphic to $\cQSym^{-}$ via the isomorphism $[\pi]_{\cDes}\mapsto F_{\cComp[\pi]}^{\cyc}$, from which the multiplication rule (\ref{e-fcycmult}) follows.
\end{proof}

As a direct consequence of Theorem \ref{t-fcycmult}, we have that $\cQSym^{-}$ is a graded $\mathbb{Q}$-subalgebra of $\QSym$. Adin et al.\ also show that the span of 
\[
\{F_{0,\emptyset}^{\cyc},F_{1,\emptyset}^{\cyc},F_{1,\{1\}}^{\cyc}\}\cup\{F_{n,[S]}^{\cyc}\}_{n\geq2,\,\emptyset\neq S\subseteq[n]},
\]
denoted $\cQSym$, is a graded $\mathbb{Q}$-subalgebra of $\QSym$, although this result is less relevant to cyclic shuffle-compatibility. Thus we have the subalgebra relations 
\[
\cQSym^{-}\subseteq\cQSym\subseteq\QSym,
\]
and $\cQSym^{-}$ is called the \textit{non-Escher subalgebra} of
$\cQSym$.

Before moving on, let us explicitly state the cyclic shuffle-compatibility of $\cDes$ as a corollary of the preceding theorem.

\begin{cor}[Cyclic shuffle-compatibility of $\cDes$] \label{c-cDes}
The cyclic descent set $\cDes$ is cyclic shuffle-compatible, and the linear map on ${\mathcal A}_{\cDes}^{\cyc}$ defined by $[\pi]_{\cDes}\mapsto F_{|\pi|,\cDes[\pi]}^{\cyc}$ is a $\mathbb{Q}$-algebra isomorphism from ${\mathcal A}_{\cDes}^{\cyc}$ to $\cQSym^{-}$.
\end{cor}

\subsection{A general cyclic shuffle-compatibility criterion for cyclic descent statistics}

The theorem below is \cite[Theorem 4.3]{Gessel2018}, which provides a necessary and sufficient condition for shuffle-compatibility of descent statistics in terms of quasisymmetric functions, and implies that the shuffle algebra of any shuffle-compatible descent statistic is a quotient algebra of $\QSym$.

\begin{thm} \label{t-scQSym}
A descent statistic $\st$ is shuffle-compatible if and only if there exists a $\mathbb{Q}$-algebra homomorphism $\phi_{\st}\colon\QSym\rightarrow A$, where $A$ is a $\mathbb{Q}$-algebra with basis $\{u_{\alpha}\}$ indexed by $\st$-equivalence classes $\alpha$ of compositions, such that $\phi_{\st}(F_{L})=u_{\alpha}$ whenever $L\in\alpha$. In this case, the linear map on ${\mathcal A}_{\st}$ defined by 
\[
\pi_{\st}\mapsto u_{\alpha},
\]
where $\Comp\pi\in\alpha$, is a $\mathbb{Q}$-algebra isomorphism from ${\mathcal A}_{\st}$ to $A$.
\end{thm}

We now prove our main result of this section: a cyclic analogue of Theorem \ref{t-scQSym}.

\begin{thm} \label{t-csccQSym}
A cyclic descent statistic $\cst$ is cyclic shuffle-compatible if and only if there exists a $\mathbb{Q}$-algebra homomorphism $\phi_{\cst}\colon\cQSym^{-}\rightarrow A$, where $A$ is a $\mathbb{Q}$-algebra with basis $\{v_{\alpha}\}$ indexed by $\cst$-equivalence classes $\alpha$ of non-Escher cyclic compositions, such that $\phi_{\cst}(F_{[L]}^{\cyc})=v_{\alpha}$ whenever $[L]\in\alpha$. In this case, the linear map on ${\mathcal A}_{\cst}^{\cyc}$ defined by 
\[
[\pi]_{\cst}\mapsto v_{\alpha},
\]
where $\cComp[\pi]\in\alpha$, is a $\mathbb{Q}$-algebra isomorphism from ${\mathcal A}_{\cst}^{\cyc}$ to $A$. 
\end{thm}

\begin{proof}
Suppose that the cyclic descent statistic $\cst$ is cyclic shuffle-compatible. Let $A={\mathcal A}_{\cst}^{\cyc}$ be the cyclic shuffle algebra of $\cst$, and let $v_{\alpha}=[\pi]_{\cst}$ for any $[\pi]$ satisfying $\cComp[\pi]\in\alpha$, so that 
\[
v_{\beta}v_{\gamma}=\sum_{\alpha}c_{\beta,\gamma}^{\alpha}v_{\alpha}
\]
where $c_{\beta,\gamma}^{\alpha}$ is the number of cyclic permutations with cyclic descent composition in $\alpha$ that are obtained as a cyclic shuffle of two disjoint cyclic permutations, one with cyclic descent composition in $\beta$ and the other with cyclic descent composition in $\gamma$. Observe that $c_{\beta,\gamma}^{\alpha}=\sum_{[L]\in\alpha}c_{J,K}^{L}$ for any choice of $[J]\in\beta$ and $[K]\in\gamma$, where $c_{J,K}^{L}$ is the number of cyclic permutations with cyclic descent composition $[L]$ that are obtained as a cyclic shuffle of two disjoint cyclic permutations, one with cyclic descent composition $[J]$ and the other with cyclic descent composition $[K]$.

Define the linear map $\phi_{\cst}\colon\cQSym^{-}\rightarrow A$ by $\phi_{\cst}(F_{[L]}^{\cyc})=v_{\alpha}$ for $[L]\in\alpha$. Then any $[J]\in\beta$ and $[K]\in\gamma$ satisfy 
\begin{align*}
\phi_{\cst}(F_{[J]}^{\cyc}F_{[K]}^{\cyc}) & =\phi_{\cst}\Big(\sum_{[L]}c_{J,K}^{L}F_{[L]}^{\cyc}\Big)\\
 & =\sum_{\alpha}\sum_{[L]\in\alpha}c_{J,K}^{L}v_{\alpha}\\
 & =\sum_{\alpha}c_{\beta,\gamma}^{\alpha}v_{\alpha}\\
 & =v_{\beta}v_{\gamma}\\
 & =\phi_{\cst}(F_{[J]}^{\cyc})\phi_{\cst}(F_{[K]}^{\cyc}),
\end{align*}
so $\phi_{\cst}$ is a $\mathbb{Q}$-algebra homomorphism, thus completing one direction of the proof. 

The converse follows from Theorem \ref{t-refine}, where we take $\cst_1$ to be $\cDes$ (which is cyclic shuffle-compatible by Corollary \ref{c-cDes}) and $\cst_2$ to be $\cst$.
\end{proof}

\begin{cor}
If $\cst$ is a cyclic shuffle-compatible descent statistic, then ${\mathcal A}_{\cst}^{\cyc}$ is isomorphic to a quotient algebra of $\cQSym^-$.
\end{cor}

To conclude this section, we state a special case of Theorem \ref{t-csccQSym} in which the homomorphism $\phi_{\cst}$ is given in terms of the homomorphism $\phi_{\st}$ of a related (linear) descent statistic; c.f.\ Theorem \ref{t-AtoAcyc}. We will use this theorem to prove cyclic shuffle-compatibility results for cyclic analogues of shuffle-compatible descent statistics.

\begin{thm} \label{t-main}
Let $\cst$ be a cyclic descent statistic and let $\st$ be a shuffle-compatible \textup{(}linear\textup{)} descent statistic, so that  there exists a $\mathbb{Q}$-algebra homomorphism $\phi_{\st}\colon\QSym\rightarrow A$ satisfying the conditions in Theorem \ref{t-scQSym}. Define the $\mathbb{Q}$-algebra homomorphism $\phi_{\cst}\colon\cQSym^{-}\rightarrow A$ by 
\[
\phi_{\cst}(F_{n,S}^{\cyc})=\sum_{i\in[n]}\phi_{\st}(F_{n,(S+i)\cap[n-1]}).
\]
Suppose that $\phi_{\cst}(F_{n,S}^{\cyc})=\phi_{\cst}(F_{n,T}^{\cyc})$ whenever $\cComp [S]$ and $\cComp [T]$ are $\cst$-equivalent cyclic compositions\textemdash so that we can write $\phi_{\cst}(F_{n,S}^{\cyc})=v_{\alpha}$ whenever $\cComp [S]\in\alpha$\textemdash and suppose that $\{v_{\alpha}\}$ is linearly independent. Then $\cst$ is cyclic shuffle-compatible and the linear map on ${\mathcal A}_{\cst}^{\cyc}$ defined by 
\[
[\pi]_{\cst}\mapsto v_{\alpha},
\]
where $\cComp[\pi]\in\alpha$, is a $\mathbb{Q}$-algebra isomorphism from ${\mathcal A}_{\cst}^{\cyc}$ to the span of $\{v_{\alpha}\}$, a subalgebra of $A$.
\end{thm}

\section{Characterizations of cyclic shuffle algebras} \label{s-charcsa}

Our next goal is to use the theory developed in the previous section to give explicit descriptions of cyclic shuffle algebras. First, let us discuss a couple statistics---the cyclic peak set $\cPk$ and the cyclic peak number $\cpk$---whose definitions were omitted from the introduction. We will then characterize the cyclic shuffle algebras of $\cPk$, $(\cpk,\cdes)$, $\cpk$, and $\cdes$. This yields new proofs for the cyclic shuffle-compatibility of the statistics $\cPk$, $\cpk$, and $\cdes$, as well as the first proof for $(\cpk,\cdes)$.

\subsection{The cyclic peak set and cyclic peak number} \label{ss-cPkcpk}

The \textit{cyclic peak set} of a linear permutation $\pi\in\mathfrak{P}_n$ is defined by 
\[
\cPk \pi \coloneqq\{\,i\in[n] : \pi_{i-1}<\pi_{i}>\pi_{i+1}\text{ where }i\text{ is considered modulo }n\,\}.
\]
and the elements of $\cPk\pi$ are called \textit{cyclic peaks} of $\pi$. Then the \textit{cyclic peak set} of a cyclic permutation $[\pi]$ is defined to be the multiset 
\[
\cPk[\pi] \coloneqq \{\{\, \cPk\bar{\pi} : \bar{\pi}\in[\pi] \,\}\}.
\]
For example, we have
$\cPk[184756] = \{\{\, \{2,4,6\}^3, \{1,3,5\}^3 \,\}\}$. It is clear from the definitions that, in general, $\cPk[\pi]$ is the multiset consisting of all cyclic shifts of $\cPk \pi$.

The \textit{cyclic peak number} of a linear permutation $\pi$ is defined by
\[
\cpk \pi \coloneqq\left|\cPk \pi \right|,
\]
and the \textit{cyclic peak number} of a cyclic permutation $[\pi]$ by
\[
\cpk[\pi] \coloneqq \cpk\pi,
\]
which is well-defined because every linear permutation in $[\pi]$ has the same number of cyclic peaks. It is easy to see that $\cPk$ and $\cpk$ are both cyclic descent statistics, so they are uniquely determined by the cyclic descent composition (equivalently, the cyclic descent set and length).

When we characterize the $(\cpk,\cdes)$ cyclic shuffle algebra, we shall need to determine all values that the $(\cpk,\cdes)$ statistic can take, which we can do with the help of two lemmas. The first of these lemmas is Proposition 2.5 of \cite{Gessel2018}, so we omit its proof.

\begin{lem} \label{l-pkdes}
Let $n\geq1$.
\begin{enumerate}
\item [\normalfont{(a)}] If $\pi\in\mathfrak{P}_{n}$, then $0\leq\pk\pi\leq\left\lfloor (n-1)/2\right\rfloor $ and $\pk\pi\leq\des\pi\leq n-\pk\pi-1$.
\item [\normalfont{(b)}] If $j$ and $k$ are integers satisfying $0\leq j\leq\left\lfloor (n-1)/2\right\rfloor $ and $j\leq k\leq n-j-1$, then there exists $\pi\in\mathfrak{P}_{n}$ with $\pk\pi=j$ and $\des\pi=k$.
\end{enumerate}
\end{lem}

\begin{lem} \label{l-cpkcdes}
Let $n\geq2$. If $\pi\in\mathfrak{P}_{n-1}$ and $m$ is greater than the largest letter of $\pi$, then $\cpk[\pi m] = \pk\pi+1$ and $\cdes[\pi m] = \des\pi+1$, where $\pi m$ is the permutation in $\mathfrak{P}_{n}$ obtained by appending the letter $m$ to $\pi$.
\end{lem}
\begin{proof}
Every peak of $\pi$ is a cyclic peak of $\pi m$, and every cyclic peak of $\pi m$ is either $m$ or a peak of $\pi$. The same relationship is true for descents of $\pi$ and cyclic descents of $\pi m$.
\end{proof}

\bco \label{c-cpkcdes} Let $n\geq2$.
\begin{enumerate}
\item [\normalfont{(a)}] If $\pi\in\mathfrak{P}_n$, then $1\le \cpk\pi \le \fl{n/2}$ and $\cpk\pi\le \cdes\pi\le n-\cpk\pi$.
\item [\normalfont{(b)}] If $j$ and $k$ are integers satisfying $1\le j\le\fl{n/2}$ and
$j\le k\le n-j$, then there exists $\pi\in\mathfrak{P}_n$ with $\cpk\pi=j$ and $\cdes\pi=k$.
\end{enumerate}
\eco

\begin{proof}
Fix $\pi\in\mathfrak{P}_n$. Let $m$ be the largest letter of $\pi$, let $\bar{\pi}$ be the unique representative of $[\pi$] which ends with $m$, and let $\pi^{\prime}$ be the permutation of length $n-1$ obtained from $\bar{\pi}$ upon removing its last letter $m$. Applying Lemma \ref{l-cpkcdes}, we obtain
$$ \cpk\pi = \cpk[\bar{\pi}] = \pk\pi^{\prime} + 1 \qquad\text{and}\qquad \cdes\pi = \cdes[\bar{\pi}] = \des\pi^{\prime} + 1.$$
Then part (a) follows from these equations and Lemma \ref{l-pkdes} (a).

To prove part (b), let $j$ and $k$ be integers in the specified ranges. By Lemma \ref{l-pkdes} (b), we know there exists a permutation $\pi^{\prime}\in\mathfrak{P}_{n-1}$ with $\pk\pi^{\prime} = j-1$ and $\des\pi^{\prime} = k-1$. Let $m\in\mathbb{P}$ be greater than the largest letter of $\pi^{\prime}$; then it follows from Lemma \ref{l-cpkcdes} that $\pi m$ is a permutation in $\mathfrak{P}_{n}$ satisfying $\cpk\pi = j$ and $\cdes\pi = k$.
\end{proof}

\subsection{The cyclic shuffle algebra of \texorpdfstring{$\cPk$}{cPk}} We will construct the cyclic shuffle algebra ${\mathcal A}_{\cPk}^{\cyc}$ from the linear shuffle algebra ${\mathcal A}_{\Pk}$. The latter is known to be isomorphic to a subalgebra $\Pi$ of $\QSym$---introduced by Stembridge \cite{Stembridge1997}---called the \textit{algebra of peaks}, which is spanned by the \textit{peak quasisymmetric functions} $K_{n,S}$ where $n$ ranges over all non-negative integers and $S$ over all possible peak sets of permutations in $\mathfrak{P}_n$. We won't need the precise definition of $K_{n,S}$ here, only that the isomorphism from ${\mathcal A}_{\Pk}$ to $\Pi$ sends $\pi_{\Pk}$ to $K_{|\pi|,\Pk \pi}$. We state this fact in the following theorem, which appears as Theorem 4.7 of \cite{Gessel2018}.

\begin{thm}[Shuffle-compatibility of $\Pk$]\label{t-Pk}
The peak set $\Pk$ is shuffle-compatible, and the linear map on ${\mathcal A}_{\Pk}$ defined by $\pi{}_{\Pk}\mapsto K_{\left|\pi\right|,\Pk\pi}$ is a $\mathbb{Q}$-algebra isomorphism from ${\mathcal A}_{\Pk}$ to $\Pi$.
\end{thm}

The analogue of Stembridge's quasisymmetric peak functions in the cyclic setting are the \textit{cyclic peak quasisymmetric functions} $K_{n,S}^{\cyc}$ recently introduced by Liang \cite{2209.00051}. Here, we shall define the cyclic peak functions $K_{n,S}^{\cyc}$ in terms of the $K_{n,S}$. For brevity, let us say that $S$ is a \textit{cyclic peak set} of $[n]$ if $S$ is the cyclic peak set of some permutation of length $n$. Then, if $S$ is a cyclic peak set of $[n]$, let
\[
K_{n,S}^{\cyc}\coloneqq \sum_{i\in[n]}K_{n,(S+i)\backslash\{1,n\}} = \sum_{\bar{\pi}\in[\pi]}K_{n,\Pk\bar{\pi}}
\]
where $\pi$ is any permutation in $\mathfrak{P}_n$ with cyclic peak set $S$. We can also write $K_{n,[S]}^{\cyc}\coloneqq K_{n,S}^{\cyc}$ since the $K_{n,S}^{\cyc}$ are invariant under cyclic shift. Liang showed that the $K_{n,[S]}^{\cyc}$ are linearly independent, and they span a subalgebra $\Lambda$ of $\cQSym$ called the \textit{algebra of cyclic peaks}.\footnote{The algebra $\Lambda$ should not be confused with another subalgebra of $\cQSym$ commonly denoted $\Lambda$: the algebra of symmetric functions.}

The following theorem---which is equivalent to Equation (5.10) of \cite{2209.00051}---gives a multiplication rule for the $K_{n,[S]}^{\cyc}$. This multiplication rule also implies that $\cPk$ is cyclic shuffle-compatible, which was first proven by Domagalski et al.\ \cite{Detal:csc} using bijective means.

\begin{thm}
\label{t-kcycmult} 
Let $m$ and $n$ be non-negative integers, let $A$ be a cyclic peak set of $[m]$, and let $B$ be a cyclic peak set of $[n]$. Then 
\begin{equation}
K_{m,[A]}^{\cyc}K_{n,[B]}^{\cyc}=\sum_{[\tau]\in[\pi]\shuffle[\sigma]}K_{m+n,\cPk[\tau]}^{\cyc}\label{e-kcycmult}
\end{equation}
where $[\pi]$ is any cyclic permutation of length $m$ with cyclic peak set $[A]$ and $[\sigma]$ is any cyclic permutation \textup{(}with $\sigma$ disjoint from $\pi$\textup{)} of length $n$ with cyclic peak set $[B]$.
\end{thm}

While Liang's proof of Theorem $\ref{t-kcycmult}$ uses enriched toric $[\vec{D}]$-partitions, we shall now use Theorem \ref{t-main} to supply an alternative proof.

\begin{proof}
First, we take $\phi_{\Pk}\colon\QSym\rightarrow\Pi$ to be the composition of the map
$F_{L}\mapsto\pi_{\Pk}$ with the map
$\pi_{\Pk}\mapsto K_{\left|\pi\right|,\Pk\pi}$ from Theorem \ref{t-Pk} where $\pi$ is any permutation with $\Pk\pi=\Pk L$; then $\phi_{\Pk}$ satisfies the conditions in Theorem \ref{t-scQSym}. 

Let $S$ be a non-Escher subset of $[n]$, and let $[P]$ be the cyclic peak set of any cyclic permutation $[\pi]$ of length $n$ with cyclic descent set $[S]$. Note that the sets $(S+i)\cap[n-1]$ where $i$ ranges from $1$ to $n$ are precisely the descent sets of the $n$ linear permutations in $[\pi]$. Hence, we have
\begin{align*}
\phi_{\cPk}(F_{n,S}^{\cyc}) & =\sum_{i\in[n]}\phi_{\Pk}(F_{n,(S+i)\cap[n-1]})=\sum_{\bar{\pi}\in[\pi]}\phi_{\Pk}(F_{n,\Des\bar{\pi}})=\sum_{\bar{\pi}\in[\pi]}K_{n,\Pk\bar{\pi}}=K_{n,[P]}^{\cyc}.
\end{align*}
Clearly, $\phi_{\cPk}(F_{n,S}^{\cyc})$ depends only on the $\cPk$-equivalence class of the cyclic composition $\cComp[S]$, and we know that the $K_{n,[P]}^{\cyc}$ are linearly independent. Applying Theorem \ref{t-main}, we conclude that $\cPk$ is cyclic shuffle-compatible and that $\mathcal{A}_{\cPk}^{\cyc}$ is isomorphic to $\Lambda$ via the isomorphism $[\pi]_{\cPk}\mapsto K_{|\pi|,\cPk[\pi]}^{\cyc}$, from which the multiplication rule \eqref{e-kcycmult} follows.
\end{proof}

\begin{cor}[Cyclic shuffle-compatibility of $\cPk$] \label{c-cPk}
The cyclic peak set $\cPk$ is cyclic shuffle-compatible, and the linear map on ${\mathcal A}_{\cPk}^{\cyc}$ defined by $[\pi]_{\cPk}\mapsto K_{\left|\pi\right|,\cPk[\pi]}^{\cyc}$ is a $\mathbb{Q}$-algebra isomorphism from ${\mathcal A}_{\cPk}^{\cyc}$ to $\Lambda$.
\end{cor}

\subsection{The cyclic shuffle algebra of \texorpdfstring{$(\cpk,\cdes)$}{(cpk, cdes)}}

We will now use Theorem \ref{t-main} to construct the cyclic shuffle algebra ${\mathcal A}_{(\cpk,\cdes)}^{\cyc}$ from the linear shuffle algebra ${\mathcal A}_{(\pk,\des)}$. We begin by recalling the following result about ${\mathcal A}_{(\pk,\des)}$, which is Theorem 5.9 of Gessel and Zhuang \cite{Gessel2018}. Below, we will use the notation $\mathbb{Q}[[t*]]$ to denote the $\mathbb{Q}$-algebra of formal power series in $t$ where the multiplication is given by the \textit{Hadamard product} $*$, defined by
\[
\Big(\sum_{n=0}^{\infty}a_{n}t^{n}\Big)*\Big(\sum_{n=0}^{\infty}b_{n}t^{n}\Big)\coloneqq\sum_{n=0}^{\infty}a_{n}b_{n}t^{n}.
\]
\begin{thm}[Shuffle-compatibility of $(\pk,\des)$] \label{t-pkdes}
\leavevmode
\begin{enumerate}
\item [\normalfont{(a)}]The pair $(\pk,\des)$ is shuffle-compatible.
\item [\normalfont{(b)}]Let 
\begin{align*}
u_{n,j,k}^{(\pk,\des)} & =\frac{t^{j+1}(y+t)^{k-j}(1+yt)^{n-j-k-1}(1+y)^{2j+1}}{(1-t)^{n+1}}x^{n}.
\end{align*}
Then the linear map on $\mathcal{A}_{(\pk,\des)}$ defined by 
\[
\pi{}_{(\pk,\des)}\mapsto\begin{cases}
u_{\left|\pi\right|,\pk\pi,\des\pi}^{(\pk,\des)}, & \text{if }\left|\pi\right|\geq1,\\
1/(1-t), & \text{if }\left|\pi\right|=0,
\end{cases}
\]
is a $\mathbb{Q}$-algebra isomorphism from $\mathcal{A}_{(\pk,\des)}$
to the span of 
\[
\left\{ \frac{1}{1-t}\right\} \bigcup\{u_{n,j,k}^{(\pk,\des)}\}_{\substack{n\geq1,\qquad\quad\;\;\;\:\\
0\leq j\leq\left\lfloor (n-1)/2\right\rfloor ,\\
j\leq k\leq n-j-1,\quad
}
},
\]
a subalgebra of $\mathbb{Q}[[t*]][x,y]$.
\end{enumerate}
\end{thm}

We note that, in the definition of $u_{n,j,k}^{(\pk,\des)}$, all products should be interpreted as ordinary multiplication; the Hadamard product in $t$ is only used when multiplying elements in the span of the $u_{n,j,k}^{(\pk,\des)}$. The same is true in Theorems \ref{t-cpkcdes}, \ref{t-cpk}, and \ref{t-cdes} presented later in this section.

\begin{thm}[Cyclic shuffle-compatibility of $(\cpk,\cdes)$] \label{t-cpkcdes}
\leavevmode
\begin{enumerate}
\item [\normalfont{(a)}]The pair $(\cpk,\cdes)$ is cyclic shuffle-compatible.
\item [\normalfont{(b)}]
Let 
\begin{align*}
v_{n,j,k}^{(\cpk,\cdes)} & =ju_{n,j-1,k}^{(\pk,\des)}+ju_{n,j-1,k-1}^{(\pk,\des)}+(k-j)u_{n,j,k-1}^{(\pk,\des)}+(n-j-k)u_{n,j,k}^{(\pk,\des)}\\
 & =[j(y+t)(1+yt)(1+y+t+yt)\\
 & \qquad\qquad\qquad\qquad+((k-j)(1+yt)+(n-j-k)(y+t))t(1+y)^{2}]\\
% & =[((n-j-k)(y+t)+(k-j)(1+yt))t(1+y)^{2}\\
% & \qquad\qquad\qquad\qquad\qquad\qquad+j(y+t)(1+yt)(1+y+t+yt)]\\
 & \qquad\quad\times\frac{t^{j}(y+t)^{k-j-1}(1+yt)^{n-j-k-1}(1+y)^{2j-1}}{(1-t)^{n+1}}x^{n}.
\end{align*}
Then the linear map on ${\mathcal A}_{(\cpk,\cdes)}^{\cyc}$ defined by 
\[
[\pi]_{(\cpk,\cdes)}\mapsto\begin{cases}
v_{\left|\pi\right|,\cpk[\pi],\cdes[\pi]}^{(\cpk,\cdes)}, & \text{if }\left|\pi\right|\geq1,\\
1/(1-t), & \text{if }\left|\pi\right|=0,
\end{cases}
\]
is a $\mathbb{Q}$-algebra homomorphism from ${\mathcal A}_{(\cpk,\cdes)}^{\cyc}$ to the span of 
\[
\left\{ \frac{1}{1-t},\frac{t(1+y)}{(1-t)^{2}}x\right\} \bigcup\{v_{n,j,k}^{(\cpk,\cdes)}\}_{n\geq2,\:1\leq j\leq\left\lfloor n/2\right\rfloor ,\:j\leq k\leq n-j},
\]
a subalgebra of $\mathbb{Q}[[t*]][x,y]$.
\item [\normalfont{(c)}]For all $n\geq2$, the $n$th homogeneous component of ${\mathcal A}_{(\cpk,\cdes)}^{\cyc}$ has dimension $\left\lfloor n^{2}/4\right\rfloor $.
\end{enumerate}
\end{thm}

\begin{proof}
We shall apply Theorem \ref{t-main} using $\st = (\pk,\des)$. In doing so, we take $\phi_{(\pk,\des)}$ to be the composition of the map 
$F_{L}\mapsto\pi_{(\pk,\des)}$ with the map 
from Theorem \ref{t-pkdes} (b), where $\pi$ is any permutation with $\pk\pi=\pk L$ and $\des\pi=\des L$. 

Let $\pi$ be a permutation of length $n\geq2$ with cyclic descent set $S$, and let $j=\cpk[\pi]$ and $k=\cdes[\pi]$ (which only depend on $S$ and not the specific choice of $\pi$). Let us consider the $n$ linear permutations in $[\pi]$, whose descent sets are given by $(S+i)\cap[n-1]$ where $i$ ranges from 1 to $n$. Among these $n$ permutations, the following hold:
\begin{itemize}
\item Exactly $j$ of these permutations have $\cpk[\pi]-1$ peaks and $\cdes[\pi]$ descents, which are those that have a cyclic peak in the first position.
\item Exactly $j$ of these permutations have $\cpk[\pi]-1$ peaks and $\cdes[\pi]-1$ descents, which are those that have a cyclic peak in the last position.
\item Exactly $k-j$ of these permutations have $\cpk[\pi]$ peaks and $\cdes[\pi]-1$ descents, which are those that have a cyclic descent in the last position which is not a cyclic peak.
\item The remaining $n-j-k$ permutations have $\cpk[\pi]$ peaks and $\cdes[\pi]$ descents.
\end{itemize}

Therefore, we have
\begin{align*}
\phi_{(\cpk,\cdes)}(F_{n,S}^{\cyc}) & =\sum_{i\in[n]}\phi_{(\pk,\des)}(F_{n,(S+i)\cap[n-1]})\\
 & = ju_{n,j-1,k}^{(\pk,\des)}+ju_{n,j-1,k-1}^{(\pk,\des)}+(k-j)u_{n,j,k-1}^{(\pk,\des)}+(n-j-k)u_{n,j,k}^{(\pk,\des)}\\
 & =v_{n,j,k}^{(\cpk,\cdes)}.
\end{align*}

For $n=0$ and $n=1$, we have 
\[
\phi_{(\cpk,\cdes)}(F_{0,\emptyset}^{\cyc})=\frac{1}{1-t}\quad\text{and}\quad\phi_{(\cpk,\cdes)}(F_{1,\emptyset}^{\cyc})=\frac{t(1+y)}{(1-t)^{2}}x.
\]
Clearly, $\phi_{(\cpk,\cdes)}(F_{n,S}^{\cyc})$ depends only on the
$(\cpk,\cdes)$-equivalence class of $\cComp[S]$. 

To prove linear independence, let us order monomials in the variables $t$ and $y$ lexicographically by the exponent of $t$ followed by the exponent of $y$, that is, $t^{a}y^{b}>t^{c}y^{d}$ if and only if either $a>c$, or if $a=c$ and $b>d$. Since Corollary \ref{c-cpkcdes} implies $j\geq1$, it is readily verified that the least monomial in $(1-t)^{n+1}v_{n,j,k}^{(\cpk,\cdes)}/x^{n}$
is $t^{j}y^{k-j}$; 
thus
\[
\left\{ \frac{(1-t)^{n+1}}{x^{n}}v_{n,j,k}^{(\cpk,\cdes)}\right\} _{\substack{1\leq j\leq\left\lfloor n/2\right\rfloor \\
j\leq k\leq n-j\:\,
}
}
\]
is linearly independent for each $n\geq2$, and this in turn implies that
\[
\left\{ \frac{1}{1-t},\frac{t(1+y)}{(1-t)^{2}}x\right\} \bigcup\{v_{n,j,k}^{(\cpk,\cdes)}\}_{\substack{n\geq2\qquad\\
\:1\leq j\leq\left\lfloor n/2\right\rfloor \\
j\leq k\leq n-j\:
}
}
\]
is linearly independent. Corollary \ref{c-cpkcdes} ensures that we have the correct limits on $j$ and $k$, so we can use Theorem \ref{t-main} to conclude that parts (a) and (b) hold.

From Corollary \ref{c-cpkcdes}, we know that for $n\geq2$, the number of $(\cpk,\cdes)$-equivalence classes of cyclic permutations of length $n$ is 
\[
\sum_{j=1}^{\left\lfloor n/2\right\rfloor }((n-j)-j+1)=\sum_{j=1}^{\left\lfloor n/2\right\rfloor }(n-2j+1),
\]
and it is straightforward to show that this is equal to $\left\lfloor n^{2}/4\right\rfloor $. Thus, part (c) follows.
\end{proof}

\subsection{The cyclic shuffle algebras of \texorpdfstring{$\cpk$}{cpk} and \texorpdfstring{$\cdes$}{cdes}}

Next, we use our characterization of the cyclic shuffle algebra ${\mathcal A}_{(\cpk,\cdes)}^{\cyc}$ along with Theorem \ref{t-refine} to characterize ${\mathcal A}_{\cpk}^{\cyc}$ and ${\mathcal A}_{\cdes}^{\cyc}$, which also provides an alternative proof for the cyclic shuffle-compatibility of $\cpk$ and $\cdes$. 

Let $\mathbb{N}$ be the set of non-negative integers. In the theorems below, we use the notation $\mathbb{Q}[x]^{\mathbb{N}}$ to denote the algebra of functions $\mathbb{N}\rightarrow\mathbb{Q}[x]$ in the non-negative integer variable $p$. For example, the map $p \mapsto \binom{p}{2}x+p^3$---which we write simply as $\binom{p}{2}x+p^3$ for brevity---is an element of $\mathbb{Q}[x]^{\mathbb{N}}$.
Moreover, in Theorem \ref{t-cpk} below, $\left(\binom{n}{k}\right)$ is the number of $k$-element multisubsets of $[n]$.

\begin{thm}[Cyclic shuffle-compatibility of $\cpk$] \label{t-cpk}
\leavevmode
\begin{enumerate}
\item [\normalfont{(a)}]The cyclic peak number $\cpk$ is cyclic shuffle-compatible.
\item [\normalfont{(b)}]The linear map on ${\mathcal A}_{\cpk}^{\cyc}$ defined by 
\begin{multline*}
\qquad [\pi]_{\cpk}\mapsto\\
\begin{cases}
{\displaystyle \frac{(\cpk[\pi](1+t)^{2}+2(\left|\pi\right|-2\cpk[\pi])t)(4t)^{\cpk[\pi]}(1+t)^{\left|\pi\right|-2\cpk[\pi]-1}}{(1-t)^{\left|\pi\right|+1}}x^{\left|\pi\right|}}, & \text{if }\left|\pi\right|\geq1,\\
1/(1-t), & \text{if }\left|\pi\right|=0,
\end{cases}
\end{multline*}
is a $\mathbb{Q}$-algebra isomorphism from ${\mathcal A}_{\cpk}^{\cyc}$ to the span of 
\[
\left\{ \frac{1}{1-t},\frac{tx}{(1-t)^{2}}\right\} \bigcup\left\{ \frac{(j(1+t)^{2}+2(n-2j)t)(4t)^{j}(1+t)^{n-2j-1}}{(1-t)^{n+1}}x^{n}\right\} _{\substack{n\geq2,\qquad\\
1\leq j\leq\left\lfloor n/2\right\rfloor 
}
},
\]
a subalgebra of $\mathbb{Q}[[t*]][x]$.
\item [\normalfont{(c)}] Let 
%\begin{align*}
%    w^{\cpk}_{n,j}&=(n-2j)2^{2j+1}\sum_{k=0}^{p-1-j}\left(\binom{n+1}{k}\right)\binom{n-2j-1}{p-1-j-k}x^{n}\\
%    &\qquad\qquad\qquad\qquad\qquad\qquad+j2^{2j}\sum_{k=0}^{p-j}\left(\binom{n+1}{k}\right)\binom{n-2j+1}{p-j-k}x^n.
%\end{align*}
\begin{align*}
    w^{\cpk}_{n,j}&=j4^{j}\sum_{k=0}^{p-j}\left(\binom{n+1}{k}\right)\binom{n-2j+1}{p-j-k} x^n\\
    &\qquad\qquad\qquad\quad+2(n-2j)4^{j}\sum_{k=0}^{p-1-j}\left(\binom{n+1}{k}\right)\binom{n-2j-1}{p-j-k-1} x^n.
\end{align*}
Then the linear map on ${\mathcal A}_{\cpk}^{\cyc}$ defined by 
\[
[\pi]_{\cpk}\mapsto \begin{cases}
w^{\cpk}_{\left|\pi\right|,\cpk[\pi]},& \text{if }\left|\pi\right|\geq1,\\
1, & \text{if }\left|\pi\right|=0,
\end{cases}
\]
is a $\mathbb{Q}$-algebra isomorphism from ${\mathcal A}_{\cpk}^{\cyc}$ to the span of 
\[
\{1\}\cup\{w^{\cpk}_{n,j}\}_{\substack{n\geq1,\qquad\\
1\leq j\leq\left\lfloor n/2\right\rfloor 
}},
\]
a subalgebra of $\mathbb{Q}[x]^{\mathbb{N}}$.
\item [\normalfont{(d)}]For all $n\geq2$, the $n$th homogeneous component
of ${\mathcal A}_{\cpk}^{\cyc}$ has dimension $\left\lfloor n/2\right\rfloor $.
\end{enumerate}
\end{thm}

\begin{proof}
Let $\phi\colon\mathcal{A}_{(\cpk,\cdes)}^{\cyc}\rightarrow\mathbb{Q}[[t*]][x]$ 
be the composition of the map from Theorem \ref{t-cpkcdes} (b) and the $y=1$ evaluation map. Since
\[
\left.v_{n,j,k}^{(\cpk,\cdes)}\right|_{y=1}=\frac{(j(1+t)^{2}+2(n-2j)t)(4t)^{j}(1+t)^{n-2j-1}}{(1-t)^{n+1}}x^{n}
\]
for all $n\geq 1$, we see that $\phi$ is precisely the map in part (b) of this theorem. Note that $v_{n,j,k}^{(\cpk,\cdes)}|_{y=1}$ depends only on $n$ and $j$, so the $v_{n,j,k}^{(\cpk,\cdes)}|_{y=1}$ correspond to $\cpk$-equivalence classes. Furthermore, it is straightforward to verify that the $v_{n,j,k}^{(\cpk,\cdes)}|_{y=1}$ are linearly independent, so we may apply Theorem \ref{t-refine} to complete the proof for parts (a), (b) and (d). Part (c) follows from part (b) and the identity
\begin{align*}
    \sum_{p=0}^\infty w^{\cpk}_{\left|\pi\right|,\cpk[\pi]}t^p&=\left(\frac{4t}{(1+t)^2}\right)^{\cpk[\pi]}\left(\frac{1+t}{1-t}\right)^{\left|\pi\right|-1}\left(\cpk [\pi]+\frac{2\left|\pi\right|t}{(1-t)^2}\right)x^{\left|\pi\right|}\\
    & =\displaystyle \frac{(\cpk[\pi](1+t)^{2}+2(\left|\pi\right|-2\cpk[\pi])t)(4t)^{\cpk[\pi]}(1+t)^{\left|\pi\right|-2\cpk[\pi]-1}}{(1-t)^{\left|\pi\right|+1}}x^{\left|\pi\right|},
\end{align*}
where the first equality follows from~\cite[Proposition 5.13 and Corollary 5.18]{2209.00051}.
\end{proof}

\begin{thm}[Cyclic shuffle-compatibility of $\cdes$] \label{t-cdes}
\leavevmode
\begin{enumerate}
\item [\normalfont{(a)}]The cyclic descent number $\cdes$ is cyclic shuffle-compatible.
\item [\normalfont{(b)}]The linear map on ${\mathcal A}_{\cdes}^{\cyc}$ defined by 
\[
[\pi]_{\cdes}\mapsto\begin{cases}
{\displaystyle \frac{\cdes[\pi]t^{\cdes[\pi]}+(\left|\pi\right|-\cdes[\pi])t^{\cdes[\pi]+1}}{(1-t)^{\left|\pi\right|+1}}x^{\left|\pi\right|}}, & \text{if }\left|\pi\right|\geq1,\\
1/(1-t), & \text{if }\left|\pi\right|=0,
\end{cases}
\]
is a $\mathbb{Q}$-algebra isomorphism from ${\mathcal A}_{\cdes}^{\cyc}$ to the span of 
\[
\left\{ \frac{1}{1-t},\frac{tx}{(1-t)^{2}}\right\} \bigcup\left\{ \frac{kt^{k}+(n-k)t^{k+1}}{(1-t)^{n+1}}x^{n}\right\} _{\substack{n\geq2,\quad\;\;\,\\
1\leq k\leq n-1
}
},
\]
a subalgebra of $\mathbb{Q}[[t*]][x]$.
\item [\normalfont{(c)}]The linear map on ${\mathcal A}_{\cdes}^{\cyc}$ defined by 
\[
[\pi]_{\cdes}\mapsto\begin{cases}
{\displaystyle {p+\left|\pi\right|-\cdes[\pi]-1 \choose \left|\pi\right|-1}px^{\left|\pi\right|}}, & \text{if }\left|\pi\right|\geq1,\\
1, & \text{if }\left|\pi\right|=0,
\end{cases}
\]
is a $\mathbb{Q}$-algebra isomorphism from ${\mathcal A}_{\cdes}^{\cyc}$ to the span of 
\[
\{1,px\}\bigcup\left\{ {p+n-k-1 \choose n-1}px^{n}\right\} _{\substack{n\geq2,\quad\;\;\,\\
1\leq k\leq n-1
}
},
\]
a subalgebra of $\mathbb{Q}[x]^{\mathbb{N}}$.
\item [\normalfont{(d)}]For all $n\geq2$, the $n$th homogeneous component of ${\mathcal A}_{\cdes}^{\cyc}$ has dimension $n-1$.
\end{enumerate}
\end{thm}
\begin{proof}
The proofs for parts (a), (b), and (d) follow in the same way as in
for Theorem \ref{t-cpk}, except that we evaluate at $y=0$ as opposed to $y=1$. Part (c) follows from part (b) and the identity
\[
\frac{kt^{k}+(n-k)t^{k+1}}{(1-t)^{n+1}}=\sum_{p=0}^{\infty}{p+n-k-1 \choose n-1}pt^{p},
\]
which was established in \cite[Lemma 5.8]{Adin2021}.
\end{proof}

\section{Cyclic permutation statistics induced by linear permutation statistics} \label{s-induced}

Recall that the cyclic permutation statistics $\cDes$ and $\cPk$ are defined by 
$$\cDes[\pi] \coloneqq \{\{\,\cDes\bar{\pi}:\bar{\pi}\in[\pi]\,\}\} \quad\text{and}\quad
\cPk[\pi] \coloneqq \{\{\,\cPk \bar{\pi}:\bar{\pi}\in[\pi]\,\}\}.$$
In other words, $\cDes[\pi]$ is simply the distribution of the linear permutation statistic $\cDes$ over all linear permutations in $[\pi]$, and similarly with $\cPk[\pi]$. In fact, any linear permutation statistic $\st$ induces a multiset-valued cyclic permutation statistic (which we also denote $\st$ by a slight abuse of notation) if we let
$$\st[\pi] \coloneqq \{\{\,\st \bar{\pi}:\bar{\pi}\in[\pi]\,\}\}.$$
In this section, we study these multiset-valued cyclic statistics induced from various linear permutation statistics.

\subsection{The cyclic statistics \texorpdfstring{$\Des$}{Des}, \texorpdfstring{$\des$}{des}, \texorpdfstring{$\Pk$}{Pk}, and \texorpdfstring{$\pk$}{pk}}

To begin, we note that the cyclic statistics induced from the linear statistics $\Des$, $\des$, $\Pk$, and $\pk$ are equivalent to $\cDes$, $\cdes$, $\cPk$, and $\cpk$, respectively.

\begin{lem} \label{l-Desequiv}
The cyclic permutation statistics $\Des$ and $\cDes$ are equivalent.
\end{lem}
\begin{proof}
Let $\pi \in \mathfrak{P}_n$. For any $\bar{\pi}\in[\pi]$, we have $\Des\bar{\pi}=\cDes\bar{\pi}\backslash\{n\}$ if $n\in\cDes\bar{\pi}$ and $\Des\bar{\pi}=\cDes\bar{\pi}$ otherwise. Therefore, we can obtain $\Des[\pi]$ from $\cDes[\pi]$ by removing every $n$ from the cyclic descent sets in $\cDes[\pi]$, and we can obtain $\cDes[\pi]$ from $\Des[\pi]$ by adding $n$ to each descent set in $\Des[\pi]$ with one fewer element than the others.
\end{proof}

\begin{lem}
\label{l-desequiv}The cyclic permutation statistics $\des$ and $\cdes$ are equivalent.
\end{lem}

\begin{proof}
Let $\pi \in \mathfrak{P}_n$. For any $\bar{\pi}\in[\pi]$, we have $\des\bar{\pi}=\cdes[\pi]-1$ if $n\in\cDes\bar{\pi}$ and $\des\bar{\pi}=\cdes[\pi]$ otherwise. The unique permutation in $[\pi]$ beginning with its largest letter does not have $n$ as a cyclic descent, so we can determine $\cdes[\pi]$ from the multiset $\des[\pi]$ by taking the largest value in $\des[\pi]$. 

Conversely, among the $n$ rotations of $\pi$, there are exactly $\cdes[\pi]$ permutations with a cyclic descent in the last position; this implies that $\des[\pi]$ is the multiset with $\cdes[\pi]$ copies of $\cdes[\pi]-1$ and $n-\cdes[\pi]$ copies of $\cdes[\pi]$, so we can determine $\des[\pi]$ from $\cdes[\pi]$ as well.
\end{proof}

\begin{lem}\label{l-Pkequiv}
The cyclic permutation statistics $\Pk$ and $\cPk$ are equivalent.
\end{lem}

\begin{proof}
Let $\pi \in \mathfrak{P}_n$. For any $\bar{\pi}\in[\pi]$, we have $\Pk\bar{\pi}=\cPk\bar{\pi}\backslash\{1\}$ if $1\in\cPk\bar{\pi}$, $\Pk\bar{\pi}=\cPk\bar{\pi}\backslash\{n\}$ if $n\in\cPk\bar{\pi}$, and $\Pk\bar{\pi}=\cPk\bar{\pi}$ otherwise. (Note that $\cPk\bar{\pi}$ cannot simultaneously contain $1$ and $n$.) Hence, we can obtain $\Pk[\pi]$ from $\cPk[\pi]$ by removing every $1$ and $n$ from the cyclic peak sets in $\cPk[\pi]$.

Conversely, suppose that we are given $\Pk[\pi]$ and wish to recover $\cPk[\pi]$. Let $i\in[n]$ be arbitrary. Notice that, among all $n$ representatives of $[\pi]$, the index of $\pi_i$ spans the entire range $\{1,2,\dots,n\}$. If $i$ is a cyclic peak of $\pi$ in particular, this means that the index of $\pi_i$ will be a peak of all $n$ representatives of $[\pi]$ except for the linear permutation beginning with $\pi_i$ and the one ending with $\pi_i$; hence, if one adds up $\pk\bar{\pi}$ over all $\bar{\pi}\in[\pi]$, then each of these $\pi_i$ will contribute $n-2$ to the summation. It follows that the sum of the sizes of all peak sets in $\Pk[\pi]$ is equal to $(n-2)\cpk[\pi]$; in other words, we can determine $\cpk[\pi]$ from $\Pk[\pi]$. It remains to show that we can recover $\cPk[\pi]$ from $\cpk[\pi]$ and $\Pk[\pi]$. To do so, we divide into two cases:
\begin{itemize}
\item \textit{Case 1}: Suppose that there exists a peak set $\Pk\bar{\pi}$ in $\Pk[\pi]$ with $\cpk[\pi]$ elements. Then $\Pk\bar{\pi} = \cPk\bar{\pi}$, and we can recover the entire multiset $\cPk[\pi]$ by taking all $n$ cyclic shifts of $\Pk\bar{\pi}$.
\item \textit{Case 2}: Suppose instead that all peak sets in $\Pk[\pi]$ have $\cpk[\pi]-1$ elements. Then, every linear permutation in $[\pi]$ has either $1$ or $n$ as a cyclic peak. In general, among the $n$ representatives of $[\pi]$, there are exactly $2\cpk[\pi]$ of them with a cyclic peak at one end. This means that $2\cpk[\pi]=n$, and since cyclic peak sets cannot contain two consecutive indices, it follows that every cyclic peak set in $\cPk[\pi]$ is of the form $\{1,3,\dots,n-1\}$ or $\{2,4,\dots,n\}$. More precisely, we must have
$$\cPk[\pi]=\{\{\, \{1,3,\dots,n-1\}^{n/2} , \{2,4,\dots,n\}^{n/2} \,\}\}.$$
\end{itemize}
Since $\cPk[\pi]$ can be recovered from $\Pk[\pi]$ in both cases, we are done.
\end{proof}

\begin{lem} \label{l-pkequiv}
The cyclic permutation statistics $\pk$ and $\cpk$ are equivalent.
\end{lem}

\begin{proof}
Let $\pi \in \mathfrak{P}_n$. As shown in the proof of Lemma \ref{l-Pkequiv}, the sum of the sizes of all peak sets in $\Pk[\pi]$ is equal to $(n-2)\cpk[\pi]$, but this is the same as the sum of all elements of the multiset $\pk[\pi]$. Thus, $\cpk[\pi]$ can be determined from $\pk[\pi]$.

For the converse, we use the observation (also used in the proof of Lemma \ref{l-Pkequiv}) that among the $n$ representatives of a cyclic permutation $[\pi]$, there are exactly $2\cpk[\pi]$ of them with a cyclic peak at one end. This implies that the multiset $\pk[\pi]$ has $2\cpk[\pi]$ copies of $\cpk[\pi]-1$ and $n-2\cpk[\pi]$ copies of $\cpk[\pi]$. Hence, $\cpk[\pi]$ completely determines $\pk[\pi]$.
\end{proof}

Since $\cDes$, $\cdes$, $\cPk$, and $\cpk$ are cyclic shuffle-compatible, it follows from these equivalences and Theorem \ref{t-equiv} that the cyclic statistics $\Des$, $\des$, $\Pk$, and $\pk$ are as well.

\begin{thm}[Cyclic shuffle-compatibility of $\Des$, $\des$, $\Pk$, and $\pk$]
The cyclic statistics $\Des$, $\des$, $\Pk$, and $\pk$ are cyclic shuffle-compatible, and we have the $\mathbb{Q}$-algebra isomorphisms
$$ \mathcal{A}_{\Des}^{\cyc} \cong \mathcal{A}_{\cDes}^{\cyc}, \quad
\mathcal{A}_{\des}^{\cyc} \cong \mathcal{A}_{\cdes}^{\cyc}, \quad
\mathcal{A}_{\Pk}^{\cyc} \cong \mathcal{A}_{\cPk}^{\cyc}, \quad\text{and}\quad
\mathcal{A}_{\pk}^{\cyc} \cong \mathcal{A}_{\cpk}^{\cyc}. $$
\end{thm}

\subsection{Symmetries revisited} \label{ss-symrev}

Let $f$ be a length-preserving involution on permutations that is both shuffle-compatibility-preserving and rotation-preserving. In
Section \ref{ss-symmetries}, we proved that if the cyclic permutation statistics $\cst_{1}$ and $\cst_{2}$ are $f$-equivalent and if $\cst_{1}$ is cyclic shuffle-compatible, then $\cst_{2}$ is also cyclic shuffle-compatible with cyclic shuffle algebra isomorphic to that of $\cst_{1}$. We now show that $f$-equivalence of two linear permutation statistics induces $f$-equivalence of their induced cyclic statistics.

\begin{lem} \label{l-fequivinduced}
Let $f$ be rotation-preserving. If $\st_{1}$ and $\st_{2}$ are $f$-equivalent linear permutation statistics, then their induced cyclic permutation statistics $\st_{1}$ and $\st_{2}$ are $f$-equivalent.
\end{lem}

\begin{proof}
Since $\st_{1}$ and $\st_{2}$ are $f$-equivalent linear permutation statistics, we have $\st_{1}\pi^{f}=\st_{1}\sigma^{f}$ if and only if $\st_{2}\pi=\st_{2}\sigma$. Suppose that $\st_{2}[\pi]=\st_{2}[\sigma]$. Then, there is a bijective correspondence $g\colon[\pi]\rightarrow[\sigma]$ satisfying $\st_{2}\bar{\pi}=\st_{2}g(\bar{\pi})$ for all $\bar{\pi}\in[\pi]$, so $\st_{1}\bar{\pi}^{f}=\st_{1}g(\bar{\pi})^{f}$ for all $\bar{\pi}\in[\pi]$. Because $f$ is rotation-preserving, the permutations $\bar{\pi}^{f}$ and $g(\bar{\pi})^{f}$ over all $\bar{\pi}\in[\pi]$ are precisely the rotations of $\pi^{f}$ and $\sigma^{f}$, respectively. Thus, we have $\st_{1}[\pi^{f}]=\st_{1}[\sigma^{f}]$. The converse follows from similar reasoning, so we have $\st_{1}[\pi^{f}]=\st_{1}[\sigma^{f}]$ if and only if $\st_{2}[\pi]=\st_{2}[\sigma]$---in other words, the cyclic permutation statistics $\st_{1}$ and $\st_{2}$ are $f$-equivalent.
\end{proof}

\begin{thm}
Let $f$ be shuffle-compatibility-preserving and rotation-preserving, and let $\st_{1}$ and $\st_{2}$ be $f$-equivalent linear permutation statistics. If the induced cyclic statistic $\st_{1}$ is cyclic shuffle-compatible, then the induced cyclic statistic $\st_{2}$ is also cyclic shuffle-compatible and $\mathcal{A}_{\st_{2}}^{\cyc}$ is isomorphic to $\mathcal{A}_{\st_{1}}^{\cyc}$.
\end{thm}

\begin{proof}
This is an immediate consequence of Theorem \ref{t-cycsym} and Lemma \ref{l-fequivinduced}. 
\end{proof}

\begin{cor} \label{c-syminduced}
Suppose that the linear permutation statistics $\st_{1}$ and $\st_{2}$ are $r$-equivalent, $c$-equivalent, or $rc$-equivalent. If the induced cyclic statistic $\st_{1}$ is cyclic shuffle-compatible, then the induced cyclic statistic $\st_{2}$ is also cyclic shuffle-compatible and its cyclic shuffle algebra $\mathcal{A}_{\st_{2}}^{\cyc}$ is isomorphic to $\mathcal{A}_{\st_{1}}^{\cyc}$.
\end{cor}

Given $\pi\in \mathfrak{P}_n$, recall that the valley set $\Val$ statistic is defined by
\[
\Val \pi \coloneqq\{\,i\in[n] : \pi_{i-1}>\pi_{i}<\pi_{i+1},\},
\]
and let us also define the \textit{cyclic valley set} $\cVal$ by
\[
\cVal \pi \coloneqq\{\,i\in[n] : \pi_{i-1}>\pi_{i}<\pi_{i+1}\text{ where }i\text{ is considered modulo }n\,\}.
\]
As a sample application of Corollary \ref{c-syminduced}, observe that $\Val$ is $c$-equivalent to $\Pk$ (as linear permutation statistics) and similarly with $\cVal$ and $\cPk$. Combining this with Lemma~\ref{l-Pkequiv}, we immediately obtain the following.

\begin{thm} [Cyclic shuffle-compatibility of $\Val$ and $\cVal$]
The cyclic statistics $\Val$ and $\cVal$ are cyclic shuffle-compatible, and we have the $\mathbb{Q}$-algebra isomorphisms
$$ \mathcal{A}_{\Val}^{\cyc} \cong \mathcal{A}_{\Pk}^{\cyc} \cong \mathcal{A}_{\cPk}^{\cyc} \cong \mathcal{A}_{\cVal}^{\cyc}. $$
\end{thm}

%One can also show that $\mathcal{A}_{\Val}^{\cyc}$ and $\mathcal{A}_{\cVal}^{\cyc}$ are isomorphic by proving that $\Val$ and $\cVal$ are equivalent statistics, which can be done in the same way as in the proof of Lemma \ref{l-Pkequiv}.

\subsection{Cyclic major index} \label{ss-cmaj}

A natural question to ask is whether there is a nice cyclic analogue of the major index. This question was raised in~\cite{Adin2021} and again in~\cite{Detal:csc}. One first needs to explain what one means by ``nice."

If $\pi\in\mathfrak{P}_m$ and $\si\in\mathfrak{P}_n$, then
$$
|\pi\shu\si| = \binom{m+n}{m}.
$$
From Stanley's theory of $P$-partitions~\cite{sta:osp}, one gets the $q$-analogue
\begin{equation}
\label{maj:gf}
    \sum_{\tau\in\pi\shu\si} q^{\maj\tau}
=q^{\maj\pi+\maj\si}\gauss{m+n}{m}
\end{equation}
where $\gauss{m+n}{m}$ is a $q$-binomial coefficient. Note that \eqref{maj:gf} implies that $\maj$ is shuffle-compatible.

It can be shown that
$$
\left|[\pi]\shu[\si]\right| = (m+n-1)\binom{m+n-2}{m-1}
$$
\cite{Detal:csc}, so one could ask that the cyclic major index give a $q$-analogue of this identity, similar to \eqref{maj:gf}, or at least for the cyclic major index to be cyclic shuffle-compatible.

Stanley also refined Equation~\eqref{maj:gf} as follows.  Let
$$
\pi\shu_k\si = \{\,\tau\in\pi\shu\si : \des\tau=k\,\}.
$$
If $\des\pi=i$ and $\des\si=j$, then
\begin{equation} \label{majdes:gf}
    \sum_{\tau\in\pi\shu_k\si} q^{\maj\tau}
=q^{\maj\pi+\maj\si+(k-i)(k-j)}
\gauss{m-j+i}{k-j}\gauss{n-i+j}{k-i};
\end{equation}
in particular, this implies that $(\des,\maj)$ is shuffle-compatible, and so we would like a $\cmaj$ statistic for which $\cmaj$ and $(\cdes,\cmaj)$ are both cyclic shuffle-compatible.

In~\cite{Adin2021}, Adin et al.\ computed the cardinality of 
$$
[\pi]\shu_k[\si]=\{\, [\tau]\in[\pi]\shu[\si] : \cdes[\tau]=k \,\}
$$
which inspired Ji and Zhang \cite{Ji2022} to define a $\cmaj$ statistic which gives a $q$-analogue of this count. They proved a generating function formula analogous to \eqref{majdes:gf}, but unfortunately, the formula does not simplify into single product, and one could hope for a different cyclic major index whose generating function would do so. Furthermore, their formula does not actually show that their $(\cdes,\cmaj)$ is cyclic shuffle-compatible; in fact, neither of their $\cmaj$ and $(\cdes,\cmaj)$ are cyclic shuffle-compatible.

Each of the cyclic statistics $\cDes$, $\cdes$, $\cPk$, and $\cpk$ is (or is equivalent to) a multiset-valued cyclic statistic induced by a corresponding linear permutation statistic, so a natural alternative definition for a cyclic major index would be to define $\cmaj$ first on linear permutations and then consider the multiset-valued statistic induced by the linear $\cmaj$. To that end, given a linear permutation $\pi$, let
$$
\cmaj\pi \coloneqq \sum_{k\in\cDes\pi} k.
$$
Unfortunately, the induced statistics $\cmaj$ and $(\cdes,\cmaj)$ are not cyclic shuffle-compatible. As a counterexample, take $\pi=1\,4\,7\,6\,9\,10\,8\,2\,5\,3$, $\sigma=1\,3\,5\,4\,7\,6\,9\,10\,8\,2$, and $\rho=11$. Then $\cdes[\pi]=\cdes[\sigma]=5$ and $\cmaj[\pi]=\cmaj[\sigma]=\{\{20,25^4,30^4,35\}\}$, but $\cmaj([\pi]\shuffle[\rho])\neq\cmaj([\sigma]\shuffle[\rho])$. For instance, 
the multiset $\{\{22,26,27,28,29,30,31,32,33,34,35\}\}$ is an element of $\cmaj([\pi]\shuffle[\rho])$ but not $\cmaj([\sigma]\shuffle[\rho])$.

Another option is to consider the cyclic statistic induced by the usual major index $\maj$, as opposed to $\cmaj$. Even if $\cmaj$ and $(\cdes,\cmaj)$ are not cyclic shuffle-compatible, it's conceivable that $\maj$ and $(\des,\maj)$ are. It turns out that $\maj$ is equivalent to $\cmaj$ and similarly with $(\des,\maj)$ and $(\cdes,\cmaj)$, so by Theorem \ref{t-equiv}, neither $\maj$ nor $(\des,\maj)$ are cyclic shuffle-compatible.

\begin{lem}\label{l-desmajequiv}
The cyclic permutation statistics $(\des,\maj)$ and $(\cdes,\cmaj)$ are equivalent.
\end{lem}

\begin{proof}
Fix a cyclic permutation $[\pi]=\{\pi=\pi^{(1)},\pi^{(2)},\dots,\pi^{(n)}\}$ of length $n$ where, for each $i\in[n]$, $\pi^{(i+1)}$ is obtained from $\pi^{(i)}$ by rotating its last element to the front of the permutation and $i$ is taken modulo $n$. We claim that, for all $i\in[n]$,
\begin{equation}
\cmaj\pi^{(i+1)}=\begin{cases}
\cmaj\pi^{(i)}+\cdes[\pi]-n, & \text{if }n\in\cDes\pi^{(i)},\\
\cmaj\pi^{(i)}+\cdes[\pi], & \text{if }n\notin\cDes\pi^{(i)}.
\end{cases}\label{e-cmaj1}
\end{equation}

To prove \eqref{e-cmaj1}, first assume that $n\in\cDes\pi^{(i)}$, and let $k=\cdes[\pi]$. Then 
$$
\cDes\pi^{(i)}= \{j_1<j_2<\cdots<j_k=n\}
$$
whereas
$$
\cDes\pi^{(i+1)}= \{1<j_1+1<j_2+1<\cdots<j_{k-1}+1\}.
$$
So
$$
\cmaj\pi^{(i)}-\cmaj\pi^{(i+1)}= n-k,
$$
which is equivalent to the first case of \eqref{e-cmaj1}. The second case is proven using a similar computation.

Observe that Equation \eqref{e-cmaj1} is equivalent to
\begin{equation}
\cmaj\pi^{(i+1)}=\maj\pi^{(i)}+\cdes[\pi],\label{e-cmaj}
\end{equation}
which allows us to determine $\cmaj[\pi]$ from $\maj[\pi]$ and $\cdes[\pi]$. Moreover, $\cdes[\pi]$ can be determined from $\des[\pi]$ by Lemma \ref{l-desequiv}, so $(\cdes,\cmaj)[\pi]$ can be determined from $(\des,\maj)[\pi]$. 

Conversely, we can use (\ref{e-cmaj}) to determine $\maj[\pi]$ from $\cmaj[\pi]$ and $\cdes[\pi]$, and $\des[\pi]$ can be determined from $\cdes[\pi]$ by Lemma \ref{l-desequiv}; altogether, this means that we can also determine $(\des,\maj)[\pi]$ from $(\cdes,\cmaj)[\pi]$.
\end{proof}

\begin{lem}
The cyclic permutation statistics $\maj$ and $\cmaj$ are equivalent.
\end{lem}
\begin{proof}
Let $\pi\in \mathfrak{P}_n$. We first claim that $\cdes[\pi]$ can be determined either from  $\maj[\pi]$ or from  $\cmaj[\pi]$. Fix $i\in\cDes\pi$. Among all $n$ representatives of $[\pi]$, the index of $\pi_i$ spans the entire range $\{1,2,\dots,n\}$. Hence, if one adds up $\maj\bar{\pi}$ over all $\bar{\pi}\in[\pi]$, then $\pi_i$ will contribute $1+2+\cdots+(n-1)=\binom{n}{2}$ to the summation. Similarly, in taking the sum of all $\cmaj\bar{\pi}$, each $\pi_i$ will contribute $1+2+\cdots+n=\binom{n+1}{2}$. Thus, the sum of all elements of the multiset $\maj[\pi]$ is equal to $\binom{n}{2}\cdes[\pi]$ and the sum of all elements of $\cmaj[\pi]$ is equal to $\binom{n+1}{2}\cdes[\pi]$, and it follows that $\cdes[\pi]$ can be determined from $\maj[\pi]$ or $\cmaj[\pi]$.

Now we are ready to prove the equivalence between $\maj$ and $\cmaj$. For one direction, $\maj[\pi]$ completely determines $\cdes[\pi]$ and hence determines $\des[\pi]$ by Lemma \ref{l-desequiv}. In addition, $\maj[\pi]$ and $\des[\pi]$ together determine $(\cdes,\cmaj)[\pi]$ by Lemma \ref{l-desmajequiv}, so $\cmaj[\pi]$ can be determined from $\maj[\pi]$. One can similarly prove the other direction using the above claim and
Lemma~\ref{l-desmajequiv}.
\end{proof}

The \textit{comajor index} $\comaj$, defined by 
$$ \comaj \pi \coloneqq \sum_{k\in\Des\pi} (n-k) $$
for $\pi \in \mathfrak{P}_n$, is a classical variation of the major index statistic. Because the linear permutation statistics $\maj$ and $\comaj$ are $rc$-equivalent and the induced cyclic statistic $\maj$ is not cyclic shuffle-compatible, it follows from Corollary \ref{c-syminduced} that the induced cyclic statistic $\comaj$ is not cyclic shuffle-compatible either. We may also define the \textit{cyclic comajor index} $\ccomaj$ by
$$ \ccomaj \pi \coloneqq \sum_{k\in\cDes\pi} (n-k) $$
for $\pi \in \mathfrak{P}_n$; then it follows similarly that the induced cyclic statistics $\ccomaj$, ($\des$,$\comaj$), and ($\cdes$,$\ccomaj$) are not cyclic shuffle-compatible either.

Perhaps surprisingly, adding just a little bit of structure to our $\cmaj$ statistic gives a statistic which is equivalent to $\cDes$. As in the proof of Lemma \ref{l-desmajequiv}, given $\pi \in \mathfrak{P}_n$, let us write
$$
[\pi] =\{ \pi=\pi^{(1)}, \pi^{(2)},\ldots, \pi^{(n)}\}
$$
where $\pi^{(i+1)}$ is obtained from $\pi^{(i)}$ by rotating its last element to the front of the permutation and $i$ is taken modulo $n$.  Define the {\em ordered cyclic major index} of $[\pi]$ to be the cyclic word
$$
\ocmaj[\pi] \coloneqq [\cmaj\pi^{(1)}, \cmaj\pi^{(2)},\ldots,  \cmaj\pi^{(n)} ],
$$
i.e., the equivalence class of the sequence $(\cmaj\pi^{(1)}, \cmaj\pi^{(2)},\ldots,\cmaj\pi^{(n)})$ under cyclic shift.
\begin{thm}
The cyclic permutation statistics $\cDes$ and $\ocmaj$ are equivalent.
\end{thm}
\begin{proof}
Let us assume throughout this proof that $n\geq2$, as the cases $n=0$ and $n=1$ are trivial. To see that $\cDes$ is a refinement of $\ocmaj$, suppose $\cDes[\pi]=\cDes[\si]$ where $\pi$ and $\sigma$ have the same length $n$. So, we can write $[\si]=\{ \si^{(1)}, \si^{(2)}, \ldots, \si^{(n)}\}$ where $\cDes\pi^{(i)}=\cDes\si^{(i)}$ for all $i\in[n]$.  It follows that
$$
\cmaj\pi^{(i)} = \sum_{k\in\cDes\pi^{(i)}} k
= \sum_{k\in\cDes\si^{(i)}} k = \cmaj\si^{(i)} 
$$
for all $i$, so $\ocmaj[\pi]=\ocmaj[\si]$.

For the converse, it is sufficient to show that the cyclic descent composition $\cComp[\pi]$ can be reconstructed from $\ocmaj[\pi]$. First, recall Equation \eqref{e-cmaj1}:
\begin{equation*}
\cmaj\pi^{(i+1)}=\begin{cases}
\cmaj\pi^{(i)}+\cdes[\pi]-n, & \text{if }n\in\cDes\pi^{(i)},\\
\cmaj\pi^{(i)}+\cdes[\pi], & \text{if }n\notin\cDes\pi^{(i)}.
\end{cases}
\end{equation*}
Since $n\geq2$, we have $1\leq\cdes[\pi]\leq n-1$, and together with the above equation, we have that $n\in\cDes\pi^{(i)}$ if and only if $\cmaj\pi^{(i)}>\cmaj\pi^{(i+1)}$.  A similar argument  shows that we can never have $\cmaj\pi^{(i)}=\cmaj\pi^{(i+1)}$.

Now, suppose we are given $\ocmaj[\pi]=[m_1,m_2,\ldots,m_n]$ where 
$m_i = \cmaj\pi^{(i)}$. Let $s$ and $t$ be two consecutive cyclic descents of $\ocmaj[\pi]$, i.e.,
$$
m_s>m_{s+1}<m_{s+2}<\cdots < m_t > m_{t+1}
$$
where subscripts are considered modulo $n$ as usual. From the previous paragraph, it follows that $n$ is in both $\cDes \pi^{(s)}$ and $\cDes\pi^{(t)}$, and that the penultimate descent in $\cDes \pi^{(s)}$ becomes the descent $n \in\cDes\pi^{(t)}$ with $n$ never being a descent for any of the intermediate cyclic descent sets. So $t-s$ (modulo $n$) is a part of the cyclic composition $\cComp[\pi]$. Therefore, all the parts of $\cComp[\pi]$ can be determined, and their order will be the same as that induced by the consecutive cyclic descents in $\ocmaj[\pi]$. Thus we have reconstructed $\cComp[\pi]$ from $\ocmaj[\pi]$, completing the proof.
\end{proof}

\begin{cor}[Cyclic shuffle-compatibility of $\ocmaj$]
The ordered cyclic major index $\ocmaj$ is cyclic shuffle-compatible, and its cyclic shuffle algebra $\mathcal{A}_{\ocmaj}^{\cyc}$ is isomorphic to $\mathcal{A}_{\cDes}^{\cyc}$.
\end{cor}

Of course, one could wonder if the unordered multiset of $\cmaj$ values is also equivalent to $\cDes$ for cyclic permutations, but this is not the case. Indeed, if the cyclic permutation statistics $\cmaj$ and $\cDes$ were equivalent, then the cyclic shuffle-compatibility of $\cDes$ would imply that $\cmaj$ is cyclic shuffle-compatible as well, which we know to be false.

\subsection{Other descent statistics} \label{ss-otherdes}
To conclude this section, let us consider the cyclic permutation statistics induced by the following linear descent statistics:
\begin{itemize}
\item The \textit{valley number} $\val$, which we defined earlier to be the number of valleys of a permutation.
\item The \textit{double descent set} $\Ddes$ and the \textit{double descent number} $\ddes$. We call $i\in \{2,3,\dots,n-1\}$ a \textit{double descent} of $\pi\in\mathfrak{P}_{n}$ if $\pi_{i-1}>\pi_{i}>\pi_{i+1}$. Then $\Ddes\pi$ is the set of double descents of $\pi$, and $\ddes\pi$ the number of double descents of $\pi$.
\item The \textit{left peak set} $\Lpk$ and the \textit{left peak number} $\lpk$. We call $i\in[n-1]$ a \textit{left peak} of $\pi\in\mathfrak{P}_{n}$ if $i$ is a peak of $\pi$, or if $i=1$ and $\pi_{1}>\pi_{2}$. Then $\Lpk\pi$ is the set of left peaks of $\pi$, and $\lpk\pi$ the number of left peaks of $\pi$.
\item The \textit{right peak set} $\Rpk$ and the \textit{right peak number} $\rpk$. We call $i\in\{2,3,\dots,n\}$ a \textit{right peak} of $\pi\in\mathfrak{P}_{n}$ if $i$ is a peak of $\pi$, or if $i=n$ and $\pi_{n-1}<\pi_{n}$. Then $\Rpk\pi$ is the set of right peaks of $\pi$, and $\rpk\pi$ the number of right peaks of $\pi$.
\item The \textit{exterior peak set} $\Epk$ and the \textit{exterior peak number} $\epk$. We call $i\in[n]$ an \textit{exterior peak} of $\pi\in\mathfrak{P}_{n}$ if $i$ is a left peak or right peak of $\pi$. Then $\Epk\pi$ is the set of exterior peaks of $\pi$, and $\epk\pi$ the number of exterior peaks of $\pi$.
\item The \textit{number of biruns} $\br$ and the \textit{number of up-down runs} $\udr$. A \textit{birun} of $\pi$ is a maximal consecutive monotone subsequence of $\pi$; an \textit{up-down run} of $\pi$ is a birun of $\pi$, or the first letter $\pi_{1}$ of $\pi$ if $\pi_{1}>\pi_{2}$. Then $\br\pi$ and $\udr\pi$ are the number of biruns and the number of up-down runs, respectively, of $\pi$.
\end{itemize}
For example, take $\pi = 713942658$. Then we have $\val\pi = 3$, $\Ddes\pi = \{5\}$, $\ddes\pi = 1$, $\Lpk\pi = \{1,4,7\}$, $\lpk\pi=3$, $\Rpk\pi = \{4,7,9\}$, $\rpk\pi=3$, $\Epk\pi = \{1,4,7,9\}$, $\epk\pi = 4$, $\br\pi = 6$, and $\udr\pi = 7$.

Aside from $\Ddes$, $\ddes$, and $\br$, all of the above statistics (as linear permutation statistics) are shuffle-compatible. Also, because these are all descent statistics, each of the induced cyclic statistics are cyclic descent statistics. Indeed, if we are given $\cDes[\pi]$ and the length of $\pi$, then we can determine $\Des[\pi]$ by Lemma \ref{l-Desequiv}, and we can then use the descent sets in $\Des[\pi]$ to obtain the multiset $\st[\pi]$ for any descent statistic $\st$.

Let us begin by examining the double descent statistics $\Ddes$ and $\ddes$. Since neither $\Ddes$ nor $\ddes$ are shuffle-compatible as linear permutation statistics, it is perhaps unsurprising that their induced cyclic statistics are not cyclic shuffle-compatible. As a counterexample, let $\pi=1234$, $\sigma=1324$, and $\rho=5$. Then both $\Ddes[\pi]=\Ddes[\sigma]$ and $\ddes[\pi]=\ddes[\sigma]$, but we have $\Ddes([\pi]\shuffle[\rho])\neq\Ddes([\sigma]\shuffle[\rho])$ and $\ddes([\pi]\shuffle[\rho])\neq\ddes([\sigma]\shuffle[\rho])$. For instance, $\{\{\emptyset^{5}\}\}$ appears three times in $\Ddes([\pi]\shuffle[\rho])$ but only twice in $\Ddes([\sigma]\shuffle[\rho])$, and accordingly $\{\{0^{5}\}\}$ appears three times in $\ddes([\pi]\shuffle[\rho])$ but only twice in $\ddes([\sigma]\shuffle[\rho])$.

While the linear statistic $\br$ is not shuffle-compatible, Domagalski et al.~\cite{Detal:csc} noted that the cyclic statistic $\cbr$ giving the number of \textit{cyclic biruns}\textemdash maximal consecutive monotone cyclic subsequences\textemdash is cyclic shuffle-compatible as it is precisely twice the number of cyclic peaks. 

\begin{thm}[Cyclic shuffle-compatibility of $\cbr$ and $(\cbr,\cdes)$]
The cyclic statistics $\cbr$ and $(\cbr,\cdes)$ are cyclic shuffle-compatible, and we have the $\mathbb{Q}$-algebra isomorphisms
$$ \mathcal{A}_{\cbr}^{\cyc} \cong \mathcal{A}_{\cpk}^{\cyc} \quad\text{and}\quad 
\mathcal{A}_{(\cbr,\cdes)}^{\cyc} \cong \mathcal{A}_{(\cpk,\cdes)}^{\cyc}. $$
\end{thm}

Because $\des$ and $\cdes$ are equivalent as cyclic permutation statistics and similarly with $\pk$ and $\cpk$, one might expect the cyclic statistics $\br$ and $\cbr$ to be equivalent as well, but this is not the case because $\br$ is not actually cyclic shuffle-compatible. For instance, consider $\pi=25673489$, $\sigma=24567389$, and $\rho=1$. Then $\br[\pi]=\br[\sigma]$, but the multiset $\{\{5^{4},6^{4},7\}\}$ appears four times in $\br([\pi]\shuffle[\rho])$ but only twice in $\br([\sigma]\shuffle[\rho])$. One can also use the same permutations $\pi$, $\sigma$, and $\rho$ to show that $(\br,\des)$ is not cyclic shuffle-compatible.

Even though the linear statistics $\Lpk$ and $\Epk$ are shuffle-compatible, their induced cyclic statistics are not cyclic shuffle-compatible. As a counterexample, take
\begin{alignat*}{2}
\pi & =11\,6\,3\,7\,1\,4\,12\,10\,2\,9\,6\,8, & \sigma & =13,\\
\pi^{\prime} & =13\,7\,2\,9\,5\,3\,10\,4\,8\,12\,6\,11,\text{ and}\qquad & \sigma^{\prime} & =1.
\end{alignat*}
Then we have $\Lpk[\pi]=\Lpk[\pi^{\prime}]$, $\Lpk[\sigma]=\Lpk[\sigma^{\prime}]$, $\Epk[\pi]=\Epk[\pi^{\prime}]$, and $\Epk[\sigma]=\Epk[\sigma^{\prime}]$, yet $\Lpk([\pi]\shuffle[\sigma])\neq\Lpk([\pi^{\prime}]\shuffle[\sigma^{\prime}])$ and $\Epk([\pi]\shuffle[\sigma])\neq\Epk([\pi^{\prime}]\shuffle[\sigma^{\prime}])$ as the multiset 
\begin{align*}
\{\{\, & \{1,5,8,11\},\{2,6,9,12\},\{3,7,10\},\{1,4,8,11\},\{2,5,9,12\},\{1,3,6,10\},\\
 & \{1,4,7,11\},\{2,5,8,12\},\{3,6,9\},\{1,4,7,10\},\{2,5,8,11\},\{1,3,6,9,12\},\{1,4,7,10\}\,\}\}
\end{align*}
belongs to $\Lpk([\pi]\shuffle[\sigma])$ but not $\Lpk([\pi^{\prime}]\shuffle[\sigma^{\prime}])$, and the multiset 
\begin{align*}
\{\{\, & \{1,4,7,10\},\{1,4,8,11\},\{2,5,8,11\},\{2,5,8,12\},\{2,5,9,12\},\\
 & \{2,6,9,12\},\{3,6,9,13\},\{3,7,10,13\},\{1,3,6,9,12\},\\
 & \{1,3,6,10,13\},\{1,4,7,10,13\},\{1,4,7,11,13\},\{1,5,8,11,13\}\,\}\}
\end{align*}
belongs to $\Epk([\pi]\shuffle[\sigma])$ but not $\Epk([\pi^{\prime}]\shuffle[\sigma^{\prime}])$.

The left peak number $\lpk$, number of up-down runs $\udr$, and the pairs $(\lpk,\des)$ and $(\udr,\des)$ are also shuffle-compatible linear statistics whose induced cyclic statistics are not cyclic shuffle-compatible. For example, take $\pi=87516439$, $\sigma=53187649$, and $\rho=2$. Then $(\lpk,\des)[\pi]=(\lpk,\des)[\sigma]$ and $(\udr,\des)[\pi]=(\udr,\des)[\sigma]$ (and thus $\lpk[\pi]=\lpk[\sigma]$ and $\udr[\pi]=\udr[\sigma]$).
However:
\begin{itemize}
\item $\{\{(3,5)^{6},(3,6)^{3}\}\}$ is in $(\lpk,\des)([\pi]\shuffle[\rho])$
but not $(\lpk,\des)([\sigma]\shuffle[\rho])$,
\item $\{\{(6,5)^{3},(6,6)^{3},(7,5)^{3}\}\}$ is in $(\udr,\des)([\pi]\shuffle[\rho])$
but not $(\udr,\des)([\sigma]\shuffle[\rho])$,
\item $\{\{3^{9}\}\}$ is in $\lpk([\pi]\shuffle[\rho])$ but not $\lpk([\sigma]\shuffle[\rho])$,
\item and $\{\{6^{6},7^{3}\}\}$ is in $\udr([\pi]\shuffle[\rho])$
but not $\udr([\sigma]\shuffle[\rho])$.
\end{itemize}
Observe that $\Rpk$ is $r$-equivalent to $\Lpk$ and $\rpk$ is $r$-equivalent to $\lpk$. Hence, by Corollary~\ref{c-syminduced}, neither $\Rpk$ nor $\rpk$ are cyclic shuffle-compatible. One can also define ``left'', ``right'', and ``exterior'' versions of the valley set and valley number; by similar symmetry arguments, none of these are cyclic shuffle-compatible either.

In contrast, the exterior peak number $\epk$ and the pair $(\epk,\des)$ are cyclic shuffle-compatible because they are equivalent to $\cpk$ and $(\cpk,\cdes)$, respectively. To prove these equivalences, we will also need to consider the \textit{cyclic valley number} statistic $\cval$: we say that $i\in[n]$ is a \textit{cyclic valley} of $\pi\in\mathfrak{S}_n$ if $\pi_{i-1}>\pi_i<\pi_{i+1}$ with the indices considered modulo $n$, and $\cval[\pi]$ is defined to be the number of cyclic valleys of any permutation in $[\pi]$. Equivalently, $\cval[\pi]$ is the cardinality of the cyclic valley set $\cVal[\pi]$ defined in Section \ref{ss-symrev}.
\begin{lem}
\label{l-valcvaleq}The cyclic permutation statistics $\val$ and $\cval$ are equivalent.
\end{lem}

\begin{proof}
We have $\val[\pi]=\pk[\pi^c]$ for all $\pi$---that is, $\val$ and $\pk$ are $c$-equivalent---and similarly with $\cval$ and $\cpk$. By Lemma \ref{l-pkequiv}, $\pk$ and $\cpk$ are equivalent, so the same is true of $\val$ and $\cval$.
\end{proof}

\begin{lem}
\label{l-cvalcpkeq}For any cyclic permutation $[\pi]$, we have $\cval[\pi]=\cpk[\pi]$.
\end{lem}

\begin{proof}
Each cyclic birun starts with a cyclic peak and ends with a cyclic valley or vice-versa. So $2\cpk[\pi]=\cbr[\pi]=2\cval[\pi]$.
\end{proof}

\begin{lem}
\label{l-epkcpkeq}The cyclic permutation statistics $\epk$ and $\cpk$ are equivalent.
\end{lem}

\begin{proof}
For any linear permutation $\pi$, we have $\epk\pi=\val\pi+1$ \cite[Lemma 2.1 (e)]{Gessel2018}, so $\epk$ and $\val$ are equivalent as linear permutation statistics and thus as cyclic permutation statistics. (We can obtain $\val[\pi]$ from $\epk[\pi]$ by subtracting 1 from each element in the multiset, and $\epk[\pi]$ from $\val[\pi]$ by adding 1 to each element.) Moreover, $\val$ is equivalent to $\cval$ (Lemma \ref{l-valcvaleq}) which is in turn equivalent to $\cpk$ (Lemma \ref{l-cvalcpkeq}); hence, $\epk$ is equivalent to $\cpk$.
\end{proof}

\begin{thm}[Cyclic shuffle-compatibility of $\val$, $\cval$, $\epk$, $(\val,\des)$, $(\cval,\cdes)$, and $(\epk,\des)$]
The cyclic statistics $\val$, $\cval$, $\epk$, $(\val,\des)$, $(\cval,\cdes)$, and $(\epk,\des)$ are cyclic shuffle-compatible, and we have the $\mathbb{Q}$-algebra isomorphisms
$$ \mathcal{A}_{\val}^{\cyc} \cong \mathcal{A}_{\cval}^{\cyc} \cong \mathcal{A}_{\epk}^{\cyc} \cong \mathcal{A}_{\cpk}^{\cyc}
\quad\text{and}\quad 
\mathcal{A}_{(\val,\des)}^{\cyc} \cong \mathcal{A}_{(\cval,\cdes)}^{\cyc} \cong \mathcal{A}_{(\epk,\des)}^{\cyc} \cong \mathcal{A}_{(\cpk,\cdes)}^{\cyc}. $$
\end{thm}

\begin{proof}
The cyclic shuffle-compatibility of $\val$, $\cval$, and $\epk$, and the corresponding isomorphisms, follow from the cyclic shuffle-compatibility of $\cpk$ and the equivalences between these four statistics. Furthermore, $(\val,\des)$ is equivalent to $(\cpk,\cdes)$ because $\val$ is equivalent to $\cpk$ and $\des$ is equivalent to $\cdes$, and similarly $(\cval,\cdes)$ and $(\epk,\des)$ are equivalent to $(\cpk,\cdes)$ as well. Because $(\cpk,\cdes)$ is cyclic shuffle-compatible, the results for $(\val,\des)$, $(\cval,\cdes)$, and $(\epk,\des)$ follow.
\end{proof}
Finally, we provide counterexamples showing that neither $(\Pk,\Val)$ nor $(\pk,\val)$ are cyclic shuffle-compatible. Let $\pi=214$, $\sigma=536$, $\pi^{\prime}=123$, and $\sigma^{\prime}=546$. Then $(\Pk,\Val)[\pi]=(\Pk,\Val)[\pi^{\prime}]$ and $(\Pk,\Val)[\sigma]=(\Pk,\Val)[\sigma^{\prime}]$, which imply $(\pk,\val)[\pi]=(\pk,\val)[\pi^{\prime}]$ and $(\pk,\val)[\sigma]=(\pk,\val)[\sigma^{\prime}]$ as well. However,
\[
\{\{\,(\emptyset,\emptyset),(\emptyset,\{5\}),(\{2\},\emptyset),(\{3\},\{2\}),(\{4\},\{3\}),(\{5\},\{4\})\,\}\}
\]
is an element of $(\Pk,\Val)([\pi]\shuffle[\sigma])$ but not $(\Pk,\Val)([\pi^{\prime}]\shuffle[\sigma^{\prime}])$, and 
\[
\{\{\,(0,0),(0,1),(1,0),(1,1)^{3}\,\}\}
\]
is an element of $(\pk,\val)([\pi]\shuffle[\sigma])$ but not $(\pk,\val)([\pi^{\prime}]\shuffle[\sigma^{\prime}])$.

\section{Open problems and questions} \label{s-openprob}

In Section \ref{s-induced}, we studied various multiset-valued cyclic statistics induced from linear statistics, as well as a $\ocmaj$---an ordered version of the $\cmaj$ statistic---which we found to be equivalent to $\cDes$. We can generalize the construction of $\ocmaj$ in the following way. Given any linear permutation statistic $\st$, let $\ost[\pi]$ be the cyclic word defined by
$$
\ost[\pi] \coloneqq [\st\pi^{(1)}, \st\pi^{(2)},\ldots, \st\pi^{(n)} ],
$$
where $[\pi] =\{ \pi=\pi^{(1)}, \pi^{(2)},\ldots, \pi^{(n)}\}$ and $\pi^{(i)}$ is defined as in Section \ref{ss-cmaj}.

\begin{problem} Study the cyclic statistics $\ost$ for various linear permutation statistics $st$.
\end{problem}

It would be interesting to find new cyclic shuffle-compatibility results stemming from these statistics---i.e., if one of the $\ost$ is cyclic shuffle-compatible and is not equivalent to another statistic already known to be shuffle-compatible. On the other hand, it would also be interesting to find nontrivial equivalences between these statistics and others, regardless of whether they are cyclic shuffle-compatible.

Next, we pose a question related to the lifting lemma of Domagalski et al.~\cite[Lemma 2.3]{Detal:csc}, which provides an avenue for proving cyclic shuffle-compatibility of a cyclic descent statistic using the shuffle-compatibility of a related linear descent statistic. The lifting lemma involves two maps $S_i$ and $M$, defined as follows. Given $\pi\in\mathfrak{S}_n$ and $i\in[n]$, let $S_i[\pi]$ be the unique linear permutation in $[\pi]$ which starts with $i$, and let $M[\pi]$ be the linear permutation of length $n-1$ obtained by first applying $S_n$ to $[\pi]$ and then removing the initial $n$. For example, we have $S_4[162453]=453162$ and $M[162453]=24531$.

\begin{lem}[Lifting lemma]
Let $\cst$ be a cyclic descent statistic and $\st$ a shuffle-compatible linear descent statistic for which the following conditions hold\textup{:}
\begin{enumerate}
\item[\normalfont{(a)}] For any $\pi,\pi^\prime\in \mathfrak{S}_n$, we have 
$$
\st(M[\pi]) = \st(M[\pi']) \qmq{implies} \cst[\pi]=\cst[\pi^\prime].
$$
\item[\normalfont{(b)}] For any $\pi,\pi^\prime\in \mathfrak{S}_n$ with $\cst[\pi]=\cst[\pi^\prime]$, there exists a bijection $f\colon [n]\rightarrow[n]$ such that 
$$
\st(S_i[\pi]) = \st(S_{f(i)}[\pi^\prime])
$$
for all $i\in[n]$.
\end{enumerate}
Then, $\cst$ is cyclic shuffle-compatible.
\end{lem}

We would like to understand how the lifting lemma fits into our algebraic framework. In particular, we have tried to prove the lifting lemma from Theorem \ref{t-AtoAcyc}, but our attempts have been unsuccessful because it is unclear to us how the conditions in the lifting lemma relate to the linear independence condition of that theorem.

\begin{question} Can the lifting lemma be proven from Theorem \ref{t-AtoAcyc}?
\end{question}

Finally, every statistic which is known to be cyclic shuffle-compatible is a cyclic descent statistic, so it is natural to ask whether any cyclic shuffle-compatible statistics are not cyclic descent statistics. In the linear setting, Gessel and Zhuang \cite{Gessel2018} had conjectured that every shuffle-compatible statistic is a descent statistic, but a counterexample was found by Kantarc{\i} O\u{g}uz \cite{KantarciOguz2018}. So, we will pose this as a question rather than as a conjecture. 

\begin{question} Is every cyclic shuffle-compatible statistic a cyclic descent statistic?
\end{question}

We note that the cyclic statistic induced by Kantarc{\i} O\u{g}uz's counterexample is not cyclic shuffle-compatible.

\vspace{15bp}

\noindent \textbf{Acknowledgements.} We thank an anonymous referee for carefully reading our paper and providing several corrections.

%\nocite{*}
\bibliographystyle{alpha}

\bibliography{csfref}

\end{document}